\documentclass[11pt]{article}

\usepackage{amsmath,amssymb,amsthm,amsfonts}
\usepackage{cite,color}
\usepackage{geometry}

\newtheorem{lemma}{Lemma}[section]
\newtheorem{prop}[lemma]{Proposition}
\newtheorem{cor}[lemma]{Corollary}
\newtheorem{thm}{Theorem}
\newtheorem{rmk}[lemma]{Remark}

\numberwithin{equation}{section}

\begin{document}

\title{\bf Threshold phenomena for symmetric decreasing solutions of
  reaction-diffusion equations}

\author{C. B. Muratov\footnote{Department of Mathematical Sciences, New
    Jersey Institute of Technology, Newark, NJ 07102, USA} \and
  X. Zhong\footnotemark[\value{footnote}]}

\maketitle

\begin{abstract}
  We study the long time behavior of solutions of the Cauchy problem
  for nonlinear reaction-diffusion equations in one space dimension
  with the nonlinearity of bistable, ignition or monostable type. We
  prove a one-to-one relation between the long time behavior of the
  solution and the limit value of its energy for symmetric decreasing
  initial data in $L^2$ under minimal assumptions on the
  nonlinearities. The obtained relation allows to establish sharp
  threshold results between propagation and extinction for monotone
  families of initial data in the considered general setting.
\end{abstract}

\section{Introduction}

In this paper, we study the Cauchy problem for the nonlinear
reaction-diffusion equation
\begin{equation}
  \label{main}
  u_t=u_{xx}+f(u),\;\;x\in\mathbb{R},\;t>0,
\end{equation}
\begin{equation}\label{initial}u(x,0)=\phi(x)\geq
  0 \;\;\;x\in\mathbb{R}, \quad \qquad \phi
  \in{L^2(\mathbb{R}) \cap L^{\infty}(\mathbb{R})}.
\end{equation}
The nonlinearity $f$ satisfies
\begin{equation}
  \label{condition}
  f \in  C^1([0,\infty)),\;\;f(0)=f(1)=0,\;\;
  f(u)<0\;\text{for}\;u>1.
\end{equation}
We are interested in the long time behavior of solution of
(\ref{main}). Since $u = 0$ and $u = 1$ are solutions of the
stationary problem for (\ref{main}), one possible behavior of the
solution is {\em extinction}, i.e.
$\lim_{t\rightarrow\infty}u(x,t)=0$ uniformly in $\mathbb{R}$. Another
possible behavior of the solution is {\em propagation}, i.e.
$\lim_{t\rightarrow\infty}u(x,t)=1$ locally uniformly in $\mathbb{R}$
and, moreover, $\lim_{t\rightarrow\infty}u(x + ct,t)=1$ locally
uniformly for all sufficiently small $c \in \mathbb R$. This type of
question was first posed in the context of combustion modeling, where
the considered initial value problem prominently appears
\cite{zeldovich38,zeldovich,buckmaster-lect}, and is also relevant to
numerous other applications in physics, chemistry and biology (see,
e.g., \cite{merzhanov99,murray,mikhailov,ko:book}).  In the context of
combustion, when cold fuel and oxidizer gases are premixed in a tube,
a sufficiently large region of heated gas generated, say, by a spark
will ignite a pair of counter-propagating flame fronts, while
insufficient heating will fail to result in ignition. Understanding
the nature of the threshold phenomena associated with ignition is,
therefore, important for many phenomena governed by reaction and
diffusion processes.

Mathematical studies of the ignition problem date back to the early
1960's. In his pioneering work, Kanel' \cite{K1964} considered the
long time behavior of solution of (\ref{main}) with ignition
nonlinearity $f$, whose initial condition $\phi$ is the characteristic
function $\chi_{[-L,L]}(x)$ of the interval $[-L,L]$.  He proved that
there exist constants $L_1 \geq L_0 > 0$, depending on $f$, such that
extinction occurs when $L<L_0$, and propagation occurs when
$L>L_1$. Aronson and Weinberger \cite{AW1975} extended this result to
bistable nonlinearities and more general initial conditions. These
works, however, did not provide any further information on the nature
of the transition between ignition and extinction.

Further insight into the ignition problem was provided very recently
by Zlato\v{s} \cite{Z2006} (see also related works
\cite{FP1997,Fl1989}), who proved that in the problem studied by
Kanel' it is possible to choose $L_0=L_1$, i.e. the transition from
extinction to propagation is {\em sharp}. He also found that the long
time behavior of the solution with the initial data corresponding to
the threshold value $L_0$ is neither extinction nor propagation.  In
particular, for bistable nonlinearities the solution of the initial
value problem with the data corresponding to $L_0$ converges to the
stationary ``bump'' solution of \eqref{main}, i.e., the unique
symmetric decreasing solution of
\begin{equation}
  \label{stationary}v''(x)+f(v(x))=0,\;\;x\in\mathbb{R}.
\end{equation}
Du and Matano \cite{DM2010} generalized the sharp transition result of
Zlato\v{s} to monotone families of compactly supported initial data by
using the zero number counting argument. By a different method,
Pol\'{a}\v{c}ik \cite{P2011} gave a higher-dimensional extension,
still for compactly supported initial data.

As was pointed out by Matano \cite{matano}, all the works on sharp
threshold behavior between ignition and extinction mentioned above
crucially rely on the assumption of the data being compactly supported
(or rapidly decaying) and, therefore, may not be applied to data that
lie in the natural function spaces, such as, e.g., $L^2(\mathbb
R)$. The purpose of this work is to provide such an extension in the
context of the problem originally considered by Kanel'. To achieve
this goal, we take advantage of the gradient flow structure of the
considered equation and develop energy-based methods that are quite
different from those used in the above works.  One of the main tools
for our analysis of the threshold behavior is the result on a
one-to-one correspondence between the long time behavior of the
solution and that of its suitably defined energy that we establish in
this paper.

As in the work of Du and Matano \cite{DM2010}, we consider an
increasing one-parameter family of initial conditions
$\phi_{\lambda}$, $\lambda > 0$, satisfying conditions in
(\ref{initial}), with $\displaystyle \lim_{\lambda \to 0} \phi_\lambda \equiv 0$,
and the map $\lambda \mapsto \phi_{\lambda}$ increasing and continuous
in the $L^2(\mathbb{R})$ norm. We also require an additional technical
assumption that $\phi_\lambda(x)$ be a symmetric decreasing function
of $x$:
\begin{enumerate}
\item[(SD)] The initial condition $\phi(x)$ in (\ref{initial}) is
  symmetric decreasing, i.e., if $\phi(-x)=\phi(x)$ and $\phi(x)$ is
  non-increasing for every $x>0$.
\end{enumerate}
This assumption allows us to avoid a possible long-time behavior
consisting of a bump solution slowly moving off to infinity, which was
pointed out for some related problems \cite{F1997}. In the case of
bistable and ignition nonlinearities (for precise definitions and
statements, see the following section) it is easy to show that if the
parameter $\lambda$ is small enough, then extinction occurs. We then
wish to know if propagation can occur when $\lambda$ is large. And a
more interesting question is: does there exist any long time behavior
of solution, which is neither extinction, nor propagation, for
intermediate values of $\lambda$?  On the other hand, for monostable
nonlinearities it is known that propagation occurs for any $\lambda >
0$ if $f'(0) > 0$ \cite{AW1975}, or even when $f(u) \sim u^p$ for
small $u$, when $p \leq p_c$, where $p_c = 3$ is the Fujita exponent
in one space dimension (see e.g. \cite{AW1978,QS2007}). Nevertheless,
the question of long-time behavior is also non-trivial for $p > p_c$
and to the best of our knowledge has not been treated so far.

Here we prove, for bistable and ignition nonlinearities, that if
propagation occurs at some value of $\lambda > 0$, then there is a
value of $\lambda = \lambda^* > 0$ which serves as a sharp threshold
between propagation for $\lambda > \lambda^*$ and extinction for
$\lambda < \lambda^*$. We also characterize the behavior of solution
at $\lambda = \lambda^*$, thus generalizing the result of Zlato\v{s}
to the considered class of data. And for monostable nonlinearities
which are supercritical with respect to the Fujita exponent, we prove
that if propagation occurs at some value of $\lambda > 0$, then there
exists a value $\lambda^* > 0$, which serves as a sharp threshold
between propagation for $\lambda > \lambda^*$ and extinction at
$\lambda \leq \lambda^*$.  Note that in this case propagation and
extinction exhaust the list of possible long-time behaviors of
solutions. In addition, we obtain a new sufficient condition for
propagation which can be easily verified.  We also note that with
minor modifications many of our conclusions still hold if $f(u)$
is only locally Lipschitz.

Our paper is organized as follows. In Section \ref{s:main} we
introduce the background results related to the variational structure
of the considered problem. Then in Section \ref{s:bistable} we
consider bistable nonlinearities and give our convergence result in
Theorem \ref{thmbistable}, our one-to-one relation result in Theorem
\ref{theorembistable} and our sharp threshold result in Theorem
\ref{sharpbistable}. Then in Section \ref{s:mono} we treat monostable
nonlinearities and give our convergence result in Theorem
\ref{thmmonostable}, our one-to-one relation result in Theorem
\ref{theoremmonostable} and our sharp threshold result in Theorem
\ref{sharpmonostable}, and in Section 5 we present results for
ignition nonlinearities, with our convergence result in Theorem
\ref{thmignition}, the relation with the limit energy in Theorem
\ref{theoremignition} and our sharp threshold result in Theorem
\ref{sharpignition}.

\section{Preliminaries}
\label{s:main}

We first recall that existence of classical solutions for \eqref{main}
with initial data satisfying \eqref{initial} is well known. In view of
\eqref{condition}, these solutions are positive, uniformly bounded
and, hence, global in time. Furthermore, it is well know that the
derivatives $u_t(x,t)$, $u_x(x,t)$, $u_{xx}(x,t)$ of the solution of
(\ref{main}) can be estimated in the uniform norm in terms of $u$
itself. More precisely, the uniform boundedness of $|u|$ in the
half-space ${t>0}$ controls the boundedness of $|u_t|$, $|u_x|$ and
$|u_{xx}|$ in the half-space $t\geq T$ for any $T > 0$ (see,
e.g. \cite{Fr1964,FM1977}). We will refer to this boundedness as
``standard parabolic regularity.'' For our purposes here, however, we
will also need a suitable existence theory for solutions in integral
norms that measure, in some sense, the rate of the decay of solutions
as $x \to \pm \infty$. This is because we wish to work with the {\em
  energy} functional, defined as
\begin{equation}
    \label{E}
    E[u] :=\int_{\mathbb{R}} \left( \frac{1}{2}u_x^2+V(u) \right) dx,
    \qquad V(u) := -\int_{0}^{u}f(s)ds.
\end{equation}
Clearly, this functional is well-defined for any $u \in
H^1(\mathbb{R}) \cap L^\infty(\mathbb R)$ and of class $C^1$ in
$H^1(\mathbb R)$. Similarly, for a given $c > 0$ we define the
exponentially weighted functional $\Phi_c$ associated with \eqref{E}
as \begin{equation}
  \label{Phic}
  \Phi_c[u] :=\int_{\mathbb{R}}e^{cx} \left( \frac{1}{2}u_x^2+V(u)
  \right) dx,
\end{equation}
which is well-defined for $L^\infty$ functions in the exponentially
weighted Sobolev space $H^1_c(\mathbb{R})$ with the norm
\begin{equation}
  \|u\|_{H^1_c}^2  :=\|u\|_{L^2_c}^2 +\|u_x\|_{L^2_c}^2, \qquad
  \|u\|_{L^2_c}^2 :=\int_{\mathbb{R}}e^{cx}u^2dx.
\end{equation}
Similarly, we can define the space $H^2_c(\mathbb R)$ as the space of
functions whose first derivatives belong to $H^1_c(\mathbb R)$.

The following proposition guarantees existence and regularity
properties of solutions of (\ref{main}) in both the usual and the
exponentially weighted Sobolev spaces.

\begin{prop}
  \label{p:exist}
  Under (\ref{condition}), there exists a unique solution $u
  \in{C^2_1(\mathbb{R}\times(0,\infty))} \cap
  L^\infty(\mathbb{R}\times(0,\infty)))$ satisfying (\ref{main}) and
  (\ref{initial}) (using the notations from \cite{E1998}), with
  $$u \in  C([0,\infty);
  L^2(\mathbb{R})) \cap C((0,\infty);H^2(\mathbb{R}))$$ and $u_t \in
  C((0,\infty);H^1(\mathbb{R}))$. Furthermore, if there exists $c>0$
  such that the initial condition $\phi(x)\in{L^2_c(\mathbb{R})}
  \cap{L^{\infty}(\mathbb{R})}$, then the solution of (\ref{main}) and
  (\ref{initial}) satisfies
  $$u \in C([0,\infty);
  L_c^2(\mathbb{R})) \cap C((0,\infty);H_c^2(\mathbb{R})),$$ with $u_t
  \in C((0,\infty);H_c^1(\mathbb{R}))$.  In addition, small variations
  of the initial data in $L^2(\mathbb R)$ result in small changes of
  solution in $H^1(\mathbb R)$ at any $t > 0$.
\end{prop}

\begin{proof}
  Follows from the arguments in the proof of \cite[Proposition
  3.1]{MN2012} based on the approach of \cite{lunardi}, taking into
  consideration that by \eqref{condition} the function $\bar u(x, t) =
  \max\{1, \| \phi \|_{L^\infty(\mathbb R)} \}$ is a universal
  supersolution for the considered problem.
\end{proof}

\begin{rmk}
  We note that Proposition \ref{p:exist} does not require hypothesis
  (SD). However, under (SD) we also have that $u(x, t)$ is a symmetric
  decreasing function of $x$ for all $t > 0$.
\end{rmk}

In view of Proposition \ref{p:exist}, by direct calculation we obtain
the well-known identity related to the energy dissipation rate for the
solutions of \eqref{main} valid for all $t > 0$:
\begin{equation}
  \label{dEdt}
  \frac{dE}{dt}[u(\cdot,t)]=-\int_{\mathbb{R}}u_t^2(x,t)dx.
\end{equation}
In fact, the basic reason for \eqref{dEdt} is the fact that
\eqref{main} is a gradient flow in $L^2$ generated by $E$.  Similarly,
as was first pointed out in \cite{M2004}, equation \eqref{main}
written in the reference frame moving with an arbitrary speed $c > 0$
is a gradient flow in $L^2_c$ generated by $\Phi_c$.  More precisely,
defining $\tilde u(x, t) := u(x + ct, t)$, which solves
\begin{equation}
  \label{referenceframe}
  \tilde u_t=\tilde u_{xx}+c\tilde u_x+f(\tilde u),
\end{equation}
it is easy to see with the help of Proposition \ref{p:exist} that an
identity similar to \eqref{dEdt} holds for $\Phi_c$:
\begin{align}
  \label{eq:dPhicdt}
  \frac{d\Phi_c}{dt}[\tilde u(\cdot,t)]=-\int_{\mathbb{R}} e^{cx}
  \tilde u_t^2(x,t)dx.
\end{align} In particular, both $E[u(\cdot, t)]$ and $\Phi_c[\tilde
u(\cdot, t)]$ are well defined and are non-increasing in $t$ for all
$t > 0$. Also note that non-trivial fixed points of
  \eqref{referenceframe} are {\em variational traveling waves}, i.e.,
  solutions that propagate with constant speed $c > 0$ invading the
  equilibrium $u = 0$ and belong to $H^1_c(\mathbb R)$
  \cite{MN2008}. Furthermore, as was shown in \cite{MN2008}, for
  sufficiently rapidly decaying front-like initial data the
  propagation speed associated with the leading edge of the solution
  (see the next paragraph for the definition) is determined by the
  special variational traveling wave solutions which are {\em
    minimizers} of $\Phi_c$ for some unique speed $c = c^\dag > 0$. In
  the context of the nonlinearities considered in this paper, the
  following proposition gives existence, uniqueness and several
  properties of these minimizers (follows directly from \cite[Theorem
  3.3]{MN2008}; in fact, under these assumptions they are the only
  variational traveling waves, see \cite[Corollary 3.4]{MN2012}).

  \begin{prop}
    \label{p:tws}
    Let $f$ satisfy \eqref{condition}, let $f'(0) \leq 0$, and let
    $u_0 = 1$ be the unique zero of $f$ such that $\int_0^{u_0} f(u)
    du < 0$. Then there exists $c^\dag > 0$ and a unique (up to
    translation) positive traveling wave solution $u(x, t) = \bar u(x
    - c^\dag t)$ of \eqref{main} such that $\bar u(+\infty) = 0$,
    $\bar u(-\infty) = 1$, $\bar u' < 0$, and $\bar u$ minimizes
    $\Phi_c$ with $c = c^\dag$.
  \end{prop}

  Turning back to the question of propagation, for a given $\delta >
  0$ we define the \emph{leading edge} $R_{\delta}(t)$ of the solution
  $u(x, t)$ of \eqref{main} as
  \begin{equation}
    R_{\delta}(t) :=\sup\{x \in \mathbb R : u(x,t)\geq\delta\} .
  \end{equation}
  If the set $\{x \in \mathbb R: u(x,t) \geq \delta\}=\varnothing$,
  then $R_{\delta}(t) :=-\infty$. Then, as follows from \cite[Theorem
  5.8]{MN2008}, under the assumptions of Proposition \ref{p:tws} for
  every $\phi \in L^2_c(\mathbb R)$ with some $c > c^\dag$, $\phi(x)
  \in [0,1]$ for all $x \in \mathbb R$, and $\displaystyle \lim_{x \to -\infty}
  \phi_(x) = 1$ the leading edge $R_\delta(t)$ propagates
  asymptotically with speed $c^\dag$ for sufficiently small $\delta >
  0$. Similarly, the same conclusion holds for the initial data
  obeying \eqref{initial}, provided that $\phi \in L^2_c(\mathbb R)$
  with some $c > c^\dag$ and $u(x, t) \to 1$ as $t \to \infty$ locally
  uniformly in $x \in \mathbb R$ \cite[Corollary 5.9]{MN2008}. In
  fact, a stronger conclusion can be made, which implies that the
  latter condition is equivalent to the stronger notion of propagation
  presented in the introduction, extending the results of Aronson and
  Weinberger \cite[Theorem 4.5]{AW1975} to the considered class of
  nonlinearities.

  \begin{prop}
    \label{p:propi}
    Under the assumptions of Proposition \ref{p:tws}, let $\phi$
    satisfy \eqref{initial} and assume that $u(x, t) \to 1$ as $t \to
    \infty$ locally uniformly in $x \in \mathbb R$. Then for every
    $\delta_0 \in (0, 1)$ and every $c \in (0, c^\dag)$, where
    $c^\dag$ is the same as in Proposition \ref{p:tws}, there exists
    $T \geq 0$ such that $R_\delta(t) \geq c t$ for every $t \geq T$
    and every $\delta \in (0, \delta_0]$.
  \end{prop}

  \begin{proof}
    Consider minimizers of $\Phi_c$ among $u \in X$, where $X$
    consists of all functions in $H^1_c(\mathbb R)$ with values in
    $[0,1]$ that vanish for all $x > 0$. We claim that a non-trivial
    minimizer $\bar u_c \in X$ of $\Phi_c$ exists for all $c \in (0,
    c^\dag)$. Indeed, by the argument in the proof of
    \cite[Proposition 5.5]{MN2008}, we have $\inf_{u \in X} \Phi_c[u]
    < 0$ for any $c \in (0, c^\dag)$. By boundedness of $u \in X$,
    $\Phi_c$ is coercive on $X$. Existence of a minimizer then follows
    from weak sequential lower semicontinuity of $\Phi_c$ on $X$ (see
    \cite[Lemma 5.3]{LMN2008}). Furthermore, by \cite[Corollary
    6.8]{LMN2008}, which can be easily seen to be applicable to $\bar
    u_c$, we have $\bar u_c(x) \to 1$ as $x \to -\infty$.

    Similarly, for large enough $R > 0$ there exists a non-trivial
    minimizer $\bar u_c^R \in X_R$ of $\Phi_c$, where $X_R$ is a
    subset of $X$ with all functions vanishing for $x < -R$ as
    well. These are stationary solutions of \eqref{referenceframe}
    with Dirichlet boundary conditions at $x = 0$ and $x = -R$, and by
    strong maximum principle we have $\bar u_c^R < 1$. Furthermore, if
    $R_n \to \infty$, then $\{ \bar u_c^{R_n} \}$ constitute a
    minimizing sequence for $\Phi_c$ in $X$ and, in view of the
    continuity of $\int_{-\infty}^0 e^{cx} V(u) dx$ with respect to
    the weak convergence in $H^1_c(\mathbb R)$ we have $\bar u_c^{R_n}
    \to \bar u_c$ strongly in $H^1_c(\mathbb R)$ and, by Sobolev
    imbedding, also locally uniformly. In particular, $\| \bar
    u_c^{R_n} \|_{L^\infty(\mathbb R)} \to 1$ as $n \to \infty$. The
    proof is then completed by using $\bar u_c^{R_n}$ with a large
    enough $n$ depending on $\delta_0$ as a subsolution after a
    sufficiently long time $t$.
 \end{proof}

 \begin{rmk}
   If in Proposition \ref{p:propi} we also have $\phi \in
   L^2_c(\mathbb R)$ for some $c > c^\dag$, then by \cite[Proposition
   5.2]{MN2008} for every $\delta_0 > 0$ and every $c' > c^\dag$ there
   exists $T \geq 0$ such that $R_\delta(t) < c' t$ for every $\delta
   \geq \delta_0$, implying that $c^\dag$ is the sharp propagation
   velocity for the level sets in the above sense. The same conclusion
   also holds for the ``trailing edge'', i.e. the leading edge defined
   using $u(-x, t)$ instead of $u(x, t)$, indicating the formation of
   a pair of counter-propagating fronts with speed $c^\dag$.
 \end{rmk}

 \begin{rmk}
   Under hypothesis (SD), the conclusion of Proposition \ref{p:propi}
   clearly implies propagation in the sense defined in the
   introduction.
 \end{rmk}

 The difficult part in applying Proposition \ref{p:propi} is to
 establish that $u(x, t) \to 1$ locally uniformly in $x \in \mathbb R$
 as $t \to \infty$ for a given initial condition $\phi(x)$. In the
 absence of such a result, we can still appeal to a weaker notion of
 propagation of the leading edge analyzed in \cite{M2004}. Following
 \cite{M2004}, we call the solution $u(x, t)$ of \eqref{main} and
 (\ref{initial}) \emph{wave-like}, if there exist constants $c > 0$
 and $T \geq 0$ such that $\phi \in L^2_c(\mathbb R)$ and
 $\Phi_c[u(\cdot,T)]<0$. Note that by monotonicity of $\Phi_c[\tilde
 u(\cdot, t)]$ and the fact that $\Phi_c[u(\cdot, t)] = e^{c^2 t}
 \Phi_c[\tilde u(\cdot, t)]$, it follows that for a wave-like solution
 we have $\Phi_c[u(\cdot, t)] < 0$ for all $t \geq T$ as well.  This
 fact allows to obtain an important characterization of the leading
 edge dynamics for wave-like solutions which is intimately related to
 the gradient descent structure of \eqref{referenceframe}. We note
 that in view of the ``hear-trigger effect'' discussed in the
 introduction in the case when $u = 0$ is linearly unstable
 \cite{AW1975}, we only need to consider the nonlinearities satisfying
 $f'(0) \leq 0$.

\begin{prop}
  \label{p:muratov}
  Let $f$ satisfy \eqref{condition}, let $f'(0) \leq 0$, and let
  $u(x,t)$ be a wave-like solution of (\ref{main}) and
  \eqref{initial}.  Then there exists a constant $\delta_0 > 0$ such
  that
  \begin{equation}
    \label{largepotential}
    V(u)\geq-\frac{1}{8} {c}^2 u^2\qquad \forall 0\leq u \leq \delta_0,
  \end{equation}
  and
  \begin{equation}
    \max_{x\in\mathbb{R}}u(x,t)\geq\delta_0,
  \end{equation}
  for $t \geq T$. Furthermore, there exists $R_0 \in \mathbb R$
    such that for every $\delta \in (0, \delta_0]$ we have
    \begin{align}
      \label{eq:Rdbelow}
      R_{\delta}(t)\geq{ct+R_0},
    \end{align}
    for all $t\geq T$.
\end{prop}
\begin{proof}
  The statement is a direct consequence of \cite[Proposition 4.10 and
  Theorem 4.11]{M2004}, which remain valid under the assumptions above
  in view of Proposition \ref{p:exist}.
\end{proof}

\noindent One of the goals of our analysis in the next sections will
be to show that under further assumptions on the nonlinearities and
hypothesis (SD) propagation in the sense of Proposition
\ref{p:muratov} implies propagation in the sense of Proposition
\ref{p:propi}.

A key ingredient of our proofs that allows us to efficiently use
variational methods and to go from sequential limits to full limits as
$t \to \infty$ without much information about the limit states relies
on an interesting observation regarding uniform H\"older continuity of
the solutions of \eqref{main} with bounded energy. This result is
stated in the following proposition. We note that a more general
result is also available in $\mathbb R^N$ (it will be discussed in
more detail elsewhere).

\begin{prop}
  \label{holder}
  Suppose that $\phi$ satisfies \eqref{initial} and $f$ satisfies
  \eqref{condition}. If $E[u(\cdot,t)]$ is bounded from below, then
  $u(x,\cdot) \in C^{1/4}([T, \infty))$ for each $x \in \mathbb R$ and
  each $T > 0$. Moreover, the corresponding H\"{o}lder constant of
  $u(x, t)$ converges to $0$ as $T \rightarrow\infty$ uniformly in
  $x$.
\end{prop}
\begin{proof}
  We denote $E_{\infty}=\displaystyle
  \lim_{t\rightarrow\infty}E[u(\cdot,t)]$. Then, using \eqref{dEdt},
  for any $x_0 \in \mathbb{R}$ and $t_2>t_1\geq T$ we have
  \begin{eqnarray}
    \int_{x_0}^{{x_0}+1}|u(x,t_2)-u(x,t_1)|dx&\leq
    &\int_{t_1}^{t_2}\int_{x_0}^{{x_0}+1}|u_t(x,t)|dxdt\nonumber\\
    &\leq&\sqrt{t_2-t_1} \left( \int_{t_1}^{t_2}\int_{x_0}^{{x_0}+1}
      u^2_t(x,t)dxdt \right)^{1/2}\nonumber\\
    &\leq&\sqrt{t_2-t_1} \left( \int_{T}^{\infty}\int_{\mathbb{R}}
      u^2_t(x,t)dxdt \right)^{1/2}\nonumber\\
    &=&\sqrt{(E[u(\cdot,T)]-E_{\infty})(t_2-t_1)}.
  \end{eqnarray}

  On the other hand, by standard parabolic regularity there exists $M
  > 0$ such that
  \begin{equation}
    \|u_x(\cdot,t) \|_{L^{\infty}(\mathbb{R})} \leq M, \quad \| u
    (\cdot,
    t) \|_{L^{\infty}(\mathbb{R})} \leq M  \qquad
    \forall t\geq T,
  \end{equation}
  Without loss of generality we can further assume that
  $u(x_0,t_2)-u(x_0,t_1)\in [0, M]$.  Then, for every $x \in I$,
    where
  \begin{equation}
    I := [x_0,x_0+\frac{u(x_0,t_2)-u(x_0,t_1)}{2M}], \qquad |I| < 1,
  \end{equation}
  we have
  \begin{equation}
    u(x,t_2) \geq u(x_0,t_2)-M(x-x_0) \geq u(x_0,t_1)+M(x-x_0) \geq
    u(x,t_1),\; x\in I.
  \end{equation}
  This implies that
   \begin{eqnarray}
    \int_{x_0}^{{x_0}+1}|u(x,t_2)-u(x,t_1)|dx &\geq&
    \int_{I}(u(x_0,t_2)-u(x_0,t_1)-2M(x-x_0))dx\nonumber\\
    &=&\frac{|u(x_0,t_2)-u(x_0,t_1)|^2}{4M}.\end{eqnarray}
    Then we have
  \begin{equation}
    |u(x_0,t_2)-u(x_0,t_1)| \leq 2\sqrt{M}
    (E[u(\cdot,T)]-E_{\infty})^{1/4}(t_2-t_1)^{1/4},
  \end{equation}
  i.e. $u(x,\cdot) \in C^{1/4}([T, \infty))$ by the arbitrariness
  of $x_0$. Moreover, the limit of the H\"{o}lder constant is
  \begin{equation}
    \lim_{T\rightarrow\infty}2\sqrt{M}
    (E[u(\cdot,T)]-E_{\infty})^{1/4}=0,
  \end{equation}
  which completes the proof.
\end{proof}

\section{Bistable Nonlinearity}
\label{s:bistable}

We now turn our attention to the study of the bistable nonlinearity,
i.e.
$f\in{C^1([0,\infty);\mathbb{R})}$, \begin{equation}\label{bistable}
  f(0)=f(\theta_0)=f(1)=0, \quad f(u)\left\{\!\!\! \begin{array}{ll}
      <0, & in \; (0,\theta_0)\cup(1,\infty),\\
      >0, & in \; (\theta_0,1), \end{array}\right. \end{equation} for
some $\theta_0\in(0,1)$. In the following, we assume an extra
condition that the $u = 1$ equilibrium is more energetically favorable
than the $u = 0$ equilibrium, i.e.,
\begin{equation}
  \label{negpotential}
  V(1)=-\int_{0}^{1}f(s)ds<0.
\end{equation}
Actually, in the context of threshold phenomena this is not a
restriction, since propagation (in the sense defined in the
introduction) becomes impossible in the opposite case. Indeed, if the
inequality opposite to \eqref{negpotential} holds, then we have $V(u)
\geq 0$ for all $u \geq 0$ and, therefore, $R_\delta \leq c t$ for any
$\delta > 0$, any $c > 0$ and large enough $t$, at least for all $\phi
\in L^2_c(\mathbb R)$ by \cite[Proposition 5.2]{MN2008}. Furthermore,
if $V(1) > 0$ and $f'(0) < 0$ (the latter condition is not essential
and may be replaced by a weaker non-degeneracy condition introduced in
the next paragraph), then the energy functional in \eqref{E} is
coercive in $H^1(\mathbb R)$, and so it is not difficult to see that
every solution of \eqref{main} and \eqref{initial} converges uniformly
to zero, implying extinction for all initial data. Thus the only case
in which the situation may be subtle is that of a {\em balanced}
bistable nonlinearity, i.e., when $V(1) = 0$, in which spreading,
i.e. sublinear behavior of the leading edge with time, namely
$R_\delta(t) \to \infty$ as $t \to \infty$, but $R_\delta(t) = o(t)$
for some $\delta > 0$, cannot be excluded a priori, even for
exponentially decaying initial data. The analysis of the balanced case
is beyond the scope of the present paper.

We further make a kind of weak non-degeneracy assumption that $f(u)
\simeq -k u^p$ for some $p \geq 1$ and $k > 0$ as $u \to 0$. More
precisely, we assume that
\begin{align}
  \label{eq:fconc}
  f'(u) \leq 0 \quad \text{for all} ~ u \in [0, \theta_1], ~ \text{for
    some} ~\theta_1 > 0,
\end{align}
and
\begin{equation}
  \label{power}
  \lim_{u \to 0} {f(u) \over u^p}  = -k \quad \text{for some} ~  p
  \geq  1 ~\text{and}~ k > 0.
\end{equation}
Note that \eqref{eq:fconc} and \eqref{power} are automatically
satisfied for the generic non-degenerate case when $f'(0) < 0$.  Under
conditions (\ref{bistable}) and (\ref{negpotential}), there exist two
roots of $V(u)$: $u=0$, $u=\theta^{\ast}\in(0,1)$, and possibly a
third root $u=\theta^{\diamond}>1$. However, since by
\eqref{condition} we have
$\displaystyle \limsup_{t\rightarrow\infty}\|u(x,t)\|_{L^{\infty}(\mathbb{R})}\leq1$,
without loss of generality, in the latter case we may suppose that
$\|\phi\|_{L^{\infty}(\mathbb{R})}<\theta^{\diamond}$. This implies
that once $u>\theta^{\ast}$, we have $V(u)<0$.

It is well known that under our assumptions (\ref{stationary})
  possesses ``bump'' solutions, i.e., classical positive solutions of
  \eqref{stationary} that vanish at infinity. After a suitable
  translation, these solutions are known to be symmetric decreasing
  and unique (see, e.g., \cite[Theorem 5]{BL1983}).  In the following
  proposition we summarize the properties of the bump solution that
  are needed for our analysis.

\begin{prop}
  \label{p:bump}
  Let $f$ satisfy conditions (\ref{condition}) and (\ref{bistable})
  through (\ref{power}), and let $v \in C^2(\mathbb R)$ be the unique
  positive symmetric decreasing solution of (\ref{stationary}). Then
  \begin{enumerate}
  \item $v(0) = \theta^\ast$ and $E_0 := E[v] >0$.
  \item If $f'(0) < 0$, we have $v(x), v'(x), v''(x) \sim e^{-\mu
      |x|}$ for $\mu = \sqrt{|f'(0)|}$ as $|x| \to \infty$.
  \item If $f'(0) = 0$, then $v(x) \sim |x|^{-{2 \over p - 1}}$,
    $v'(x) \sim |x|^{-{p+1 \over p - 1}}$, $v''(x) \sim |x|^{-{2 p
        \over p - 1}}$ as $|x| \to \infty$.
  \item $f'(v(\cdot)) \in L^1(\mathbb R)$ and $v' \in H^1(\mathbb R)$
  \end{enumerate}
\end{prop}

\begin{proof}
  The fact that $v(0) = \theta^\ast$ follows from \cite[Theorem
  5]{BL1983}. Integrating \eqref{stationary} once, we obtain $|v'| =
  \sqrt{2 V(v)}$, where by the previous result the constant of
  integration is zero.  Upon second integration we arrive at
  \begin{align}
    \label{eq:vint}
    |x| = \int_v^{\theta^\ast} {d u \over \sqrt{2 V(u)}}.
  \end{align}
  The proof then follows by a careful analysis of the singularity in
  the integral in \eqref{eq:vint} to establish the decay of the
  solution. Once the decay is known, the rest of the statements
  follows straightforwardly.
\end{proof}

Our main theorems in this section are about the following
convergence and equivalence conclusions.

\begin{thm}
  \label{thmbistable}
  Let $f$ satisfy conditions (\ref{condition}) and (\ref{bistable})
  through (\ref{power}). Let $\phi(x)$ satisfy condition
  (\ref{initial}) and hypothesis (SD). Then one of the following
  holds.
  \begin{enumerate}
  \item   $\lim_{t\rightarrow\infty}u(x,t)=1$ locally uniformly in
    $\mathbb{R}$,
  \item $\lim_{t\rightarrow\infty}u(x,t)=v(x)$ uniformly in
    $\mathbb{R}$,
  \item $\lim_{t\rightarrow\infty}u(x,t)=0$ uniformly in
  $\mathbb{R}$.
  \end{enumerate}
\end{thm}

\noindent We will prove Theorem \ref{thmbistable} together with
establishing the following one-to-one relation between the long time
behavior of the solutions and those of their energy $E$.

\begin{thm}
  \label{theorembistable}
  Under the same assumptions as in Theorem \ref{thmbistable}, we have
  the following three alternatives:
  \begin{enumerate}
  \item $\displaystyle \lim_{t\rightarrow\infty}u(x,t)=1$ locally uniformly in
  $\mathbb{R}$ ${\Leftrightarrow}$
  $\displaystyle \lim_{t\rightarrow\infty}E[u(\cdot,t)]=-\infty$.
\item $\displaystyle \lim_{t\rightarrow\infty}u(x,t)=v(x)$ uniformly
  in $\mathbb{R}$ ${\Leftrightarrow}$ $\displaystyle
  \lim_{t\rightarrow\infty}E[u(\cdot,t)]=E_0$.
\item $\displaystyle \lim_{t\rightarrow\infty}u(x,t)=0$ uniformly in
  $\mathbb{R}$ ${\Leftrightarrow}$ $\displaystyle
  \lim_{t\rightarrow\infty}E[u(\cdot,t)]=0$.
  \end{enumerate}
\end{thm}

The strategy of our proof is as follows. We wish to show that the
limit behaviors of the energy in Theorem \ref{theorembistable} are the
only possible ones. So we first prove that if $E[u(\cdot,t)]$ is not
bounded from below, then $u$ converges to $1$ locally uniformly. And
the reverse also holds. Then for bounded from below $E[u(\cdot,t)]$,
the solution $u(x, t)$ converges to either $0$ or $v(x)$. Finally, the
convergence of $u(x, t)$ to $0$ or $v(x)$ implies the corresponding
convergence of energy.

Let us begin by assuming that $E[u(\cdot,t)]$ is not bounded from
below. In this case, for cubic nonlinearity Flores proved in
\cite{Fl1989} that $\displaystyle \lim_{t\rightarrow\infty}u(x,t)=1$ locally
uniformly by constructing a proper subsolution. Under (SD), we will
prove a stronger conclusion. We will prove that if there exists
$T\geq0$ such that $E[u(\cdot,T)]<0$, then propagation occurs, in the
sense defined in the introduction.  Throughout the rest of this
section, the assumptions of the above theorems are always assumed to
be satisfied, and $u(x, t)$ always refers to the solutions of
\eqref{main} and \eqref{initial}.

\begin{lemma}
  \label{wavelike}
  Suppose there exists $c_0>0$ such that
  $\phi(x)\in{H^1_{c_0}(\mathbb{R})}$. If there exists $T\geq0$ such
  that $E[u(\cdot,T)]<0$, then $u(x,t)$ is wave-like.
\end{lemma}
\begin{proof}
  First observe that if $\phi(x)\in{H_{c_0}^1(\mathbb{R})}$, then
  $u(\cdot, T)\in{H^1(\mathbb{R})} \cap{H_{c_0}^1(\mathbb{R})}$. Then
  for any small $\varepsilon>0$, if $E[u(x,T)]=-\varepsilon<0$ there
  exists $L>0$ such that $V(u(x,T))\geq0$ for $|x|\geq{L}$, and
  \begin{eqnarray}
    \int_{\{x\leq-L\}} \left( \frac{1}{2}u_x^2(x,T)+V(u(x,T)) \right)
    dx<\frac{\varepsilon}{4}, \\
    \int_{\{x\geq{L}\}}e^{c_0x} \left(
      \frac{1}{2}u_x^2(x,T)+V(u(x,T)) \right) dx
    <\frac{\varepsilon}{4}.
  \end{eqnarray}
  Note that if we use smaller positive $c$ instead $c_0$ in the above
  inequality, the inequality still holds. And by the definition of $L$
  we know that
  \begin{equation}
    \int_{\{|x|<L\}} \left( \frac{1}{2}u_x^2(x,T)+V(u(x,T)) \right)
    dx<-\varepsilon.
  \end{equation}
  So we can find a sufficiently small $c \in (0, c_0)$ such that
  \begin{equation}
    \int_{\{|x|<L\}} e^{c x} \left( \frac{1}{2}u_x^2(x,T)+V(u(x,T))
    \right) dx  < -\frac{\varepsilon}{2},
  \end{equation}
  and
  \begin{equation}
    \Phi_{c}[u(\cdot,T)]= \int_{\mathbb{R}}e^{c x} \left(
      \frac{1}{2}u_x^2(x,T)+V(u(x,T)) \right) dx<0.
  \end{equation}
  So $u$ is wave-like.
\end{proof}

We next show that for symmetric decreasing solutions and bistable
nonlinearities the wave-like property also implies propagation in the
sense of the introduction.

\begin{lemma}
  \label{spread}
  Suppose that $u(x, t)$ is wave-like. Then
  $\displaystyle \lim_{t\rightarrow\infty}u(x,t)=1$ locally uniformly in
  $\mathbb{R}$.
\end{lemma}
\begin{proof}
  In view of the definition of $\theta^{\ast}$ we have $\delta_0 >
  \theta^\ast$ in Proposition \ref{p:muratov}. Therefore, by that
  proposition
  \begin{equation}
    R_{\theta^{\ast}}(t)>\frac{c t}{2},
  \end{equation}
  for sufficiently large $t$. Then, because $u(x,t)$ is symmetric
  decreasing, for any $L>0$ there exists $T_L>0$ such that
  $u(x,t)>\theta^{\ast}$ on the interval $[-L,L]$, for any
  $t\geq{T_L}$. Now, consider $\underline u(x, t)$ solving
  \eqref{main} with $\underline u(x, T_L) = \theta^\ast$ for all $x
  \in (-L, L)$ and $\underline u(\pm L, t) = \theta^\ast$ for all $t >
  T_L$. Since by our assumption on the nonlinearity the function
  $\underline u(x, T_L)$ is a strict subsolution, in the spirit of
  \cite[Proposition 2.2]{AW1975} we have $\underline u(\cdot, t) \to
  v_L$ uniformly on $[-L, L]$, where $v_L$ solves \eqref{stationary}
  with $v_L(\pm L) = \theta^\ast$. Then, by comparison principle we
  obtain
    \begin{align}
      \label{eq:uto1}
      v_L \leq \liminf_{t \to \infty} u(\cdot, t) \leq \limsup_{t \to
        \infty} u(\cdot, t) \leq 1 \qquad \text{uniformly in} ~ [-L,
      L].
    \end{align}
    Also, by standard elliptic estimates we have $v_L \to \bar v$
    locally uniformly as $L \to \infty$, where $\bar v$ solves
    \eqref{stationary} in the whole of $\mathbb R$. Since by
    construction $\bar v \geq \theta^\ast$, we have in fact $\bar v =
    1$. Then, passing to the limit in \eqref{eq:uto1}, we obtain the
    result.
\end{proof}

Our next Lemma uses a truncation argument to extend the conclusion of
Lemma \ref{spread} to solutions that are not necessarily lying in any
exponentially weighted Sobolev space, but have negative energy at some
time $T \geq 0$.

\begin{lemma}
  \label{test}
  Suppose that there exists $T\geq0$ such that $E[u(\cdot,T)]<0$, then
  $\displaystyle \lim_{t\rightarrow\infty}u(x,t)=1$ locally uniformly in
  $\mathbb{R}$.
\end{lemma}

\begin{proof}
  For any $L>0$, we construct a cutoff function
  $\varphi_{L}(x)=\eta(|x|/L)$, where $\eta$ is a non-increasing
  $C^{\infty}(\mathbb{R})$ function such that $\eta(x)=1$ for $x<1$,
  and $\eta(x)=0$ for $x>2$. Let
  $\hat{\phi}(x;L)=\varphi_{L}(x)u(x,T)$, so that $\hat{\phi}(x;L) \to
  u(x, T)$ in $H^1(\mathbb R)$ as $L \to \infty$. By our assumption
  and continuity of $E$, there exists a sufficiently large $L=L_0$,
  such that $E[\hat{\phi}(x;L_0)]<0$. Note that $\hat{\phi}(x;L_0)$ is
  a compactly supported function, so it lies in $H^1_c(\mathbb{R})$
  for any $c > 0$. Now consider the solution $\hat{u}(x,t)$ which
  satisfies (\ref{main}) with initial condition
  $\hat{u}(x,0)=\hat{\phi}(x;L_0)$. From Lemma \ref{spread}, we know
  that $\displaystyle \lim_{t\rightarrow\infty}\hat{u}(x,t)=1$ locally
  uniformly in $\mathbb{R}$. So by comparison principle $u(x,t+T) \geq
  \hat{u}(x,t)$, which proves the lemma.
\end{proof}

\noindent An obvious corollary to the above lemma is the following.

\begin{cor}
  \label{c:Einf}
  Suppose that $\displaystyle
  \lim_{t\rightarrow\infty}E[u(\cdot,t)]=-\infty$, then $\displaystyle
  \lim_{t\rightarrow\infty}u(x,t)=1$ locally uniformly in
  $\mathbb{R}$.
\end{cor}

Our next lemma provides a sufficient condition for propagation, which,
in particular, yields a conclusion converse to that of Corollary
\ref{c:Einf}.
\begin{lemma}
  \label{unboundedenergy}
  Suppose that $\displaystyle \lim_{t\rightarrow\infty}u(x,t)=1$
  locally uniformly in $\mathbb{R}$, then $\displaystyle
  \lim_{t\rightarrow\infty}E[u(\cdot,t)]=-\infty$.
\end{lemma}
\begin{proof}
  We argue by contradiction. Suppose that
  $\displaystyle \lim_{t\rightarrow\infty}u(x,t)=1$ locally uniformly in
  $\mathbb{R}$ and $E[u(\cdot,t)]$ is bounded below. Then for any
  $L>0$, we can construct a cutoff function
  $\kappa_{L}(x)=\eta(|x|-L)$, where $\eta$ is defined in the proof of
  Lemma \ref{test}. For any $L>0$, $\kappa_{L}(x)=1$ for $|x|<L+1$,
  $\kappa(x)=0$ for $|x|>L+2$, and $|\kappa_{L}'(x)|$ is
  bounded. Since $u(x,t)$ is symmetric decreasing, $\kappa_{L}$,
  $\kappa_{L}'$ are both bounded, and $u$, $u_x$ are both bounded for
  all $t\geq1$ by standard parabolic regularity, for $\tilde u_L(x, t)
  :=\kappa_L(x)u(x,t)$ with any $t\geq1$ we have the following energy
  estimate:
  \begin{align}
    E[\tilde u_L(x, t)]& =2\int^{L+1}_{0}V(u)dx+ \int^{L+1}_{0}u_x^2dx
    \nonumber \\
    & + \int^{L+2}_{L+1} \left\{ \left(
        \frac{\partial(\kappa_{L}u)}{\partial x}
      \right)^2+2V(\kappa_{L}u) \right\} dx \leq
    2\int^{L+1}_{0}V(u)dx+C,
  \end{align}
  where the constant $C$ is independent of $L$.  Since $\displaystyle
  \lim_{t\rightarrow\infty}u(x,t)=1$ locally uniformly in
  $\mathbb{R}$, for every $L_0>0$ satisfying $(L_0+1)V(1)<-C$, we can
  choose $t_0>0$ such that $V(u(x,t_0))<V(1)/2<0$ for any
  $x\in(-L_0-1,L_0+1)$.  This implies that
  $\tilde{\phi}(x;L_0)=\kappa_{L_0}(x)u(x,t_0)$ satisfies
  $E[\tilde{\phi}(x;L_0)]<0$.

  Note that $\tilde{\phi}(x;L_0)$ is a compactly supported function,
  so it lies in $H^1_c(\mathbb{R})$ for any $c > 0$. Now consider the
  solution $\tilde{u}(x,t)$ that satisfies (\ref{main}), with the
  initial condition $\tilde{u}(x,0)=\tilde{\phi}(x;L_0)$. By
  Proposition \ref{p:muratov}, Lemma \ref{wavelike}, and the fact that
  \begin{equation}
    u(x,t+t_0)\geq\tilde{u}(x,t),\;\;x \in \mathbb{R},\;t>0,
  \end{equation}
  there exists $c > 0$ such that for any $t>t_0$,
  \begin{equation}
    R_{\theta^{\ast}}>c(t-t_0)+R_0,
  \end{equation}
  for some constant $R_0 \in \mathbb R$. Moreover, we can find
  $T_0>0$ such that for any $t>T_0$ and $|x|\leq{c t /2}$ we
    have
  \begin{equation}
    u(x,t+t_0)\geq\theta^{\ast}.
  \end{equation}

  On the other hand, by \eqref{dEdt} there exists a sufficiently large
  $t_{\alpha} \geq 0$ such that
  \begin{equation}
    \int^{\infty}_{t_{\alpha}}\int_{\mathbb{R}}u^2_t(x,t)dxdt<{\alpha}^2,
  \end{equation}
  for every $\alpha > 0$. Let us take $\alpha=\theta_{0} \sqrt{c} /9$,
  $t_1 > \max \{t_0,t_{\alpha}\}$ and $x_1=R_{\theta_0/2}(t_1)$. We
  also take $T>T_0$ such that $x_1<c T/4$, and $t_2=t_1+T$, $x_2=x_1+c
  T$. Then by Cauchy-Schwarz inequality we have
  \begin{eqnarray}\int^{t_2}_{t_1}\int^{x_2}_{x_1}|u_t(x,t)|dxdt
    &\leq&\sqrt{(x_2-x_1)(t_2-t_1)} \left( \int^{t_2}_{t_1}
      \int^{x_2}_{x_1}u^2_t(x,t)dxdt \right)^{1/2}\nonumber\\
    &\leq&\sqrt{c} T \left( \int^{\infty}_{t_{\alpha}}
      \int_{\mathbb{R}}u^2_t(x,t)dxdt \right)^{1/2}\nonumber\\
    &\leq&\frac{c T\theta_0}{9}. \label{eq:utcontr}
  \end{eqnarray}

  At the same time, since by construction $0<x_1<c T/4$, we also have
  \begin{eqnarray}
    \int^{t_2}_{t_1}\int^{x_2}_{x_1}|u_t(x,t)|dxdt& \geq
    &\int^{c T/2}_{c T/4} \left( \int^{t_2}_{t_1}|u_t(x,t)|dt \right)
    dx \nonumber\\
    &\geq&\int^{c T/2}_{c T/4}(u(x,t_2)-u(x,t_1))dx.
  \end{eqnarray}
  Since $t_2>T>T_0$, we have $u(x,t_2)\geq\theta^{\ast}>\theta_{0}$
  for $x\in(c T/4,c T/2)$. And by the definition of $x_1$ and $T$, we
  have $u(x,t_1)<\theta_{0}/2$ for $x\in(c T/4,c T/2)$. So we have
  \begin{equation}
    \int^{t_2}_{t_1}\int^{x_2}_{x_1}|u_t(x,t)|dxdt \geq \frac{c T\theta_0}{8},
  \end{equation}
  which contradicts \eqref{eq:utcontr}.
\end{proof}

Note that we have just proved the equivalence in part 1 of Theorem
\ref{theorembistable}. Indeed, we have a stronger corollary.
\begin{cor}\label{spreadcondition}
  We have $\displaystyle \lim_{t\rightarrow\infty}u(x,t)=1$ locally
  uniformly in $\mathbb{R}$, if and only if there exists $T\geq0$ such
  that $\displaystyle \lim_{t\rightarrow\infty}E[u(\cdot,T)]<0$.
\end{cor}

We now turn our attention to the case when $E[u(\cdot,t)]$ is bounded
from below. By Lemmas \ref{test} and \ref{unboundedenergy},
boundedness of $E[u(\cdot, t)]$ implies $\displaystyle
\lim_{t\rightarrow\infty}E[u(\cdot,t)]\geq0$. Below we prove that in
this case either $\displaystyle \lim_{t\rightarrow\infty}u(x,t)=0$
uniformly in $\mathbb{R}$, or $\displaystyle
\lim_{t\rightarrow\infty}u(x,t)=v(x)$ uniformly in $\mathbb{R}$.  The
idea of our proof is due to Fife \cite[Lemma 10]{Fi1979}.  We refine
Fife's arguments under our weaker assumptions on the nonlinearity and
(SD).

The next Lemma establishes existence of an increasing sequence $\{ t_n
\}$ tending to infinity on which the solution converges to a zero of
$V(u)$ at the origin, thus allowing only two possibilities for the
value of $\displaystyle \lim_{n \to \infty} u(0, t_n)$.

\begin{lemma}
  \label{origin}
  If $E[u(\cdot,t)]$ is bounded from below, there exists an increasing
  sequence $\{t_n\}$ with $\displaystyle
  \lim_{n\rightarrow\infty}t_n=\infty$ such that either $\displaystyle
  \lim_{n\rightarrow\infty}u(0,t_n)=0$, or $\displaystyle
  \lim_{n\rightarrow\infty}u(0,t_n)=\theta^{\ast}$.
\end{lemma}
\begin{proof}
  We multiply $u_x$ on both sides of equation (\ref{main}), and
  integrate the products over $(-\infty,0)$. Then we have
  \begin{eqnarray}
    \label{intleft}
    \int_{-\infty}^{0}u_x(x,t)u_t(x,t)dx&=
    &\int_{-\infty}^{0}(u_{xx}(x,t)+f(u(x,t)))u_x(x,t)dx\nonumber\\
    &=& \left. \frac{1}{2}u_x^2(x,t)
    \right|_{x=-\infty}^{0}-(V(u(0,t))-V(u(-\infty,t)))
    \nonumber\\
    &=&-V(u(0,t)).
  \end{eqnarray}
  From monotonicity of $u$ on $(-\infty,0)$ and standard parabolic
  regularity, for $t \geq 1$ the left-hand side of (\ref{intleft}) can
  be controlled by
  \begin{eqnarray}
    \left| \int_{-\infty}^{0}u_x(x,t)u_t(x,t)dx \right| &\leq
    &\|u_t(\cdot,t)\|_{L^2(-\infty,0)}\|u_x(\cdot,t)\|_{L^2(-\infty,0)}
    \nonumber\\
    &\leq&\|u_t(\cdot,t)\|_{L^2(\mathbb{R})}
    \|u_x(\cdot,t)\|^{1/2}_{L^\infty(\mathbb R)} |u(0, t)|^{1/2}  \nonumber\\
    &\leq&\|u_t(\cdot,t)\|_{L^2(\mathbb{R})}
    \|u_x\|^{1/2}_{L^\infty(\mathbb R \times (1,\infty))} \max \{ 1, \|
    \phi \|^{1/2}_{L^\infty(\mathbb R)} \}.
  \end{eqnarray}
  where we applied Cauchy-Schwarz inequality in the first line.  Since
  $E[u(\cdot,t)]$ is bounded from below, by \eqref{dEdt} we have
  \begin{equation}
    \int_{1}^{\infty} \int_\mathbb R u_t^2(x,t)dx dt <\infty.
  \end{equation}
  Therefore, there exists an unbounded increasing sequence $\{t_n\}$
  such that $\displaystyle \lim_{n\rightarrow\infty} \|u_t(\cdot,t_n)
  \|_{L^2(\mathbb{R})}=0$. Since also $\| u_x \|_{L^\infty(\mathbb R
    \times (1, \infty))} < \infty$ by standard parabolic regularity,
  this implies that $\displaystyle \lim_{n\rightarrow\infty}V(u(0,t_n))=0$ by
  \eqref{intleft}. Furthermore, since $\displaystyle \limsup_{n \to
    \infty} \| u(\cdot, t_n) \|_{L^\infty(\mathbb R)} \leq 1$, by the
  assumptions on the nonlinearity either
  $\displaystyle \lim_{t\rightarrow\infty}u(0,t_n)=0$ or
  $\displaystyle \lim_{t\rightarrow\infty}u(0,t_n)=\theta^{\ast}$.
\end{proof}

\begin{rmk}
  \label{r:tnl1}
  The sequence $\{t_n\}$ in Lemma \ref{origin} satisfies $\|
  u_t(\cdot, t_n) \|_{L^2(\mathbb R)} \to 0$ and can be chosen so as
  $t_{n+1}-t_{n}\leq1$ for every $n$.
\end{rmk}

Our next result treats the first alternative in Lemma \ref{origin}.

\begin{lemma}
  \label{tozero}
  Suppose that there exists an increasing sequence $\{t_n\}$ such that
  $\displaystyle \lim_{n\rightarrow\infty}t_n=\infty$, and
  $\displaystyle \lim_{n\rightarrow\infty}u(0,t_n)=0$, then
  $\displaystyle \lim_{t\rightarrow\infty}u(x,t)=0$ uniformly in
  $\mathbb{R}$.
\end{lemma}
\begin{proof}
  Recall that the maximum of solution $u$ is always at the origin. By
  the structure of the nonlinearity $f$, we know that once
  $\max_{x{\in}\mathbb{R}}u(x,T)<\theta_0$ for some $T\geq0$, then
  $\displaystyle \lim_{t\rightarrow\infty}u(x,t)=0$ uniformly in $\mathbb{R}$.
\end{proof}

Combining the results of Lemma \ref{origin} and Lemma \ref{tozero}, we
now prove the following result.

\begin{lemma}
  \label{tobump}
  Suppose that $E[u(\cdot,t)]$ is bounded from below in $t$, then
  either $\displaystyle \lim_{t\rightarrow\infty}u(x,t)=0$, or
  $\displaystyle \lim_{t\rightarrow\infty}u(x,t)=v(x)$, uniformly in
  $\mathbb{R}$.
\end{lemma}
\begin{proof}
  From Lemma \ref{origin} and Lemma \ref{tozero}, we only need to
  prove that if the increasing sequence $\{t_n\}$ in Lemma
  \ref{origin} satisfies $\displaystyle
  \lim_{n\rightarrow\infty}u(0,t_n)=\theta^{\ast}$, then
  $\displaystyle \lim_{t\rightarrow\infty}u(x,t)=v(x)$ uniformly in
  $\mathbb{R}$. To prove this, we first prove the locally uniform
  convergence on the sequence $\{t_n\}$. Let $w(x,t) := u(x,t)-v(x)$,
  then in view of $v(0) = \theta^\ast$ by Proposition \ref{p:bump} we
  have
  \begin{equation}
    \label{eq:ivp}
    w_t=w_{xx}+f'(\tilde{u})w,\;\;w_x(0,t)=0,
    \;\;w(0,t)=u(0,t)-\theta^{\ast},
  \end{equation}
  where $\tilde{u}$ is between $u$ and $v$.  We claim that
  \begin{equation}
    \lim_{n\rightarrow\infty}w(x,t_n)=0,
  \end{equation}
  locally uniformly in $\mathbb{R}$. The proof follows from the
  continuous dependence on the data for solutions of the initial value
  problem in $x$ obtained from \eqref{eq:ivp} for each $t = t_n$
  fixed. Indeed, at $t=t_n\geq1$ we denote $w_n(x) := w(x, t_n)$,
  $g_n(x) :=u_t(x,t_n)$, $K_n(x) :=f'(\tilde{u}(x,t_n))$, $\alpha_n :=
  u(0,t_n)-\theta^{\ast}$, and consider \eqref{eq:ivp} as an ordinary
  differential equation in $x > 0$:
  \begin{equation}
    w_n''=g_n - K_n w_n,\;\;\; w_n'(0)=0,\;\;w_n(0)=\alpha_n.
  \end{equation}
  For any $L>0$, by integration over $(0,L)$ and an application of
  Cauchy-Schwarz inequality we have
  \begin{eqnarray}
    \max_{0 \leq x{\leq}L}|w_n'(x)|
    &\leq&\sqrt{L}\|g_n\|_{L^2(\mathbb{R})}+
    L\|K_n\|_{L^{\infty}(\mathbb{R})}\max_{0{\leq}x \leq L}|w_n(x)|\nonumber\\
    &\leq&\sqrt{L}\|g_n\|_{L^2(\mathbb{R})}
    +L\mathcal{K}\max_{0{\leq}x{\leq}L}|w_n(x)|, \label{eq:maxwnp}
\end{eqnarray}
where the constant $\mathcal{K}$ satisfies
\begin{align}
  |f'(s)| \leq \mathcal{K},\;\;0\leq{s}\leq
  \max\{1,\|\phi(x)\|_{L^{\infty}(\mathbb{R})}\}.
\end{align}

For fixed $L>0$, we choose a sufficiently large integer $l$ such that
$2{\delta}L\mathcal{K}\leq1$ for $\delta :=L/l$. We next take
\begin{align}
  W_{n,k} & :=\max_{(k-1)\delta\leq{x}\leq{k\delta}} |w_n(x)|, \quad k
  \in \mathbb N, \\
  m_{n,0} & :=\alpha_n,\qquad m_{n,k} :=\max_{1 \leq k' \leq k} W_{n,k'}.
\end{align}
Then $m_{n,k}$ is non-decreasing in $k$, and $m_{n,k}= \displaystyle
\max_{0\leq{x}\leq{k\delta}}|w_n(x)|$.  By \eqref{eq:maxwnp} and our
choice of $\delta$, for any $1 \leq k \leq l$ we have
\begin{eqnarray}
  m_{n,k}-m_{n,k-1}&\leq&\delta\max_{0\leq{x}\leq L}| w'_n(x)|
  \nonumber\\
  &\leq& \delta(\sqrt{L}\|g_n\|_{L^2(\mathbb{R})}+
  L\mathcal{K}m_{n,k})\nonumber\\
  &\leq& \delta\sqrt{L}\|g_n\|_{L^2(\mathbb{R})}+\frac{m_{n,k}}{2}.
\end{eqnarray}
This implies that for any $1 \leq k \leq l$ we have
\begin{equation}
  m_{n,k}\leq2m_{n,k-1}+G_n,
\end{equation}
where $G_n := 2 \delta\sqrt{L}\|g_n\|_{L^2(\mathbb{R})}$. Since by
definition $m_{n,0}=\alpha_n$, by iteration and symmetry of $w_n(x)$
we have
\begin{equation}
  \max_{-L{\leq}x{\leq}L}|w_n(x)|=m_{n,l} \leq 2^{l}\alpha_n+(2^{l}-1)G_n.
\end{equation}

Now, as $n\rightarrow\infty$, by Lemma \ref{origin} and Remark
\ref{r:tnl1} we know that $u(0,t_n)-\theta^{\ast}\rightarrow0$ and
$\|u_t(x,t_n)\|_{L^2(\mathbb{R})}\rightarrow0$, so that
$\alpha_n\rightarrow0$, $G_n\rightarrow0$, and $\displaystyle
\max_{-L\leq{x}\leq{L}}|w_n(x)|\rightarrow0$, i.e. $u(x,t_n)$
converges to $v(x)$ locally uniformly. Then by Proposition
\ref{holder} and the fact that by Remark \ref{r:tnl1} the sequence
$\{t_n\}$ can be chosen so as $t_{n+1}-t_n\leq1$, we can obtain the
full limit convergence. Indeed, since the H\"{o}lder constant in $t$
of $u(x, t)$ converges to $0$ as $n \rightarrow \infty$ uniformly for
all $|x|\leq{L}$ and all $t_n<t<t_{n+1}$, we have
\begin{equation}
  |u(x,t)-v(x)|\leq|u(x,t_n)-v(x)|+|u(x,t)-u(x,t_n)|\rightarrow0
  \;\text{as}\;n\rightarrow\infty.
\end{equation}

Finally, let us prove that convergence of $u(x, t)$ to $v(x)$ is, in
fact, uniform. Indeed, since $u(x,t)$ is symmetric decreasing in $x$
and $v(x)\rightarrow0$ as $|x|\rightarrow\infty$, for any $L>0$, $t>0$
we have
\begin{eqnarray}
  \sup_{|x|\geq{L}}|w(x,t)|&=
  &\sup_{|x|\geq{L}}|u(x,t)-v(x)|\nonumber\\
&\leq&\max_{|x|\geq{L}}\{u(x,t),v(x)\}\nonumber\\
&\leq&\max\{u(L,t),v(L)\}\nonumber\\
&\leq&v(L)+\max_{|x|\leq{L}}|w(x,t)|.
\end{eqnarray}
This implies that
\begin{equation}
  \sup_{x\in{\mathbb{R}}}|w(x,t)|\leq{v(L)}+\max_{|x|\leq{L}}|w(x,t)|.
\end{equation}
Then, for any $\varepsilon>0$ we can find $L>0$ sufficiently large
such that $v(L)<\varepsilon/2$. We can also find $T>0$ such that
$|w(x,t)|<\varepsilon/2$ for any $x \in [-L,L]$, $t>T$. So we get
\begin{equation}
  \lim_{t\rightarrow\infty}|w(x,t)|=0,
\end{equation}
uniformly in $x \in \mathbb{R}$, which proves the lemma.
\end{proof}

Note that in view of the results in the preceding lemmas, by proving
Lemma \ref{tobump} we have just proved Theorem \ref{thmbistable}.

\begin{rmk}
  \label{w1}
  By standard parabolic regularity, under the assumptions of Lemma
  \ref{tobump} we also have
  \begin{equation}
    \lim_{t\rightarrow\infty}u(x,t)=v(x)\;in\;C^1(\mathbb{R}).
  \end{equation}
\end{rmk}

We now turn to the study of the limit value of energy. At first, we
prove that the energy of the solution goes to zero, if extinction
occurs.

\begin{lemma}
  \label{etozero}
  If $\displaystyle \lim_{t\rightarrow\infty}u(x,t)=0$ uniformly in $\mathbb{R}$,
  then $\displaystyle \lim_{t\rightarrow\infty}E[u(\cdot,t)]=0$.
\end{lemma}
\begin{proof}
  From condition (SD), we have
  \begin{eqnarray}
    \int_{\mathbb{R}}\frac{1}{2}u^2_x(x,t)dx&=
    &\int_0^{\infty}u^2_x(x,t)dx\nonumber\\
    &\leq&\|u_x(x,t)\|_{L^{\infty}(\mathbb{R})} \, u(0,t).
  \end{eqnarray}
  By standard parabolic regularity, if $\displaystyle
  \lim_{t\rightarrow\infty}u(x,t)=0$ uniformly in $\mathbb{R}$, then
  \begin{equation}
    \lim_{t\rightarrow\infty}\int_{\mathbb{R}}
    \frac{1}{2}u^2_x(x,t)dx\rightarrow0.
  \end{equation}
  So we only need to show that
  $\displaystyle \lim_{t\rightarrow\infty}\int_{\mathbb{R}}V(u(x,t))dx=0$.

  If $f'(0) < 0$, there exists $C>0$ such that
  \begin{equation}
    0\leq{V(u)}\leq{Cu^2},
  \end{equation}
  for small enough $u$. Then from the usual energy estimate we obtain
  $\displaystyle
  \lim_{t\rightarrow\infty}\|u(\cdot,t)\|^2_{L^2(\mathbb{R})}=0$
  exponentially, so that $\displaystyle
  \lim_{t\rightarrow\infty}\int_{\mathbb{R}}V(u(x,t))dx=0$ as well.
  Alternatively, if $f'(0)=0$, then by \eqref{power} we have
  \begin{equation}
    0\leq{V(u)}\leq{Cu^{p+1}},
  \end{equation}
  for some $C > 0$ and sufficiently small $u$. So it is enough to show
  that $\displaystyle
  \lim_{t\rightarrow\infty}\|u(\cdot,t)\|^{p+1}_{L^{p+1}(\mathbb{R})}=0$. In
  view of \eqref{power} we can use the solution $\bar u(x, t)$ of the
  heat equation:
  \begin{align}
    \bar u_t = \bar u_{xx}, \; x \in \mathbb R, ~t > T, \qquad \bar
    u(x, T) = u(x, T), \; x \in \mathbb R,
  \end{align}
  as a supersolution to obtain (see, e.g., \cite[Proposition
  48.4]{QS2007})
  \begin{align}
    \label{eq:uto0}
    \|u(\cdot,t)\|_{L^{p+1}(\mathbb{R})} \leq \|\bar
    u(\cdot,t)\|_{L^{p+1}(\mathbb{R})} \leq (4 \pi t)^{-{ p - 1
        \over 4(p + 1)}} \| u(\cdot, T) \|_{L^2(\mathbb R)} \to 0
    ~\text{as}~ t \to \infty,
  \end{align}
and the statement follows.
\end{proof}

If, on the other hand, $\displaystyle \lim_{t\rightarrow\infty}u(x,t)=v(x)$
uniformly in $\mathbb{R}$, then we claim that $E[u(\cdot,t)]$ has a
limit as $t\rightarrow\infty$, and the value of the limit is equal to
$E_0$ defined in Proposition \ref{p:bump}. We begin with the
analysis of the non-degenerate case.

\begin{lemma}
  \label{etoeo1}
  Suppose that $f'(0)<0$, then $\displaystyle
  \lim_{t\rightarrow\infty}u(x,t)=v(x)$ uniformly in $\mathbb{R}$
  implies $\displaystyle \lim_{t\rightarrow\infty}E[u(\cdot,t)]=E_0$.
\end{lemma}
\begin{proof}
  At first, we show that for any fixed $L>0$, the energy
  $E[u(\cdot,t);L]$ of $u(x,t)$ restricted to $[-L,L]$, namely
  $E[u(\cdot,t);L] := \int_{-L}^L \left( \frac12 u_x^2 + V(u) \right)
  dx$, converges to the energy $E[v;L]$ of $v(x)$ restricted to
  $[-L,L]$. Then we show that $E[u(\cdot,t)]-E[u(\cdot,t);L]$
  converges to $E_0-E[v;L]$ for sufficiently large $L$.\par

  Since upon integration of \eqref{stationary} we have $|v'| = \sqrt{2
    V(u)}$, on the interval $[-L,L]$ we can compute
  \begin{equation}
    E[v;L] = \int_{-L}^{L} \left( \frac{1}{2}|v'|^2
      +V(v) \right)
    dx=2\sqrt{2}\int_{v(L)}^{\theta^{\ast}}\sqrt{V(u)}du.
  \end{equation}
  We also know that $u(x,t) \rightarrow v(x)$, $u_x(x,t) \rightarrow
  v'(x)$ uniformly in $x\in[-L,L]$, as $t\rightarrow\infty$, by Lemma
  \ref{w1}. This implies that
  \begin{equation}
    \lim_{t\rightarrow\infty}E[u(\cdot,t);L]=E[v;L].
  \end{equation}
  By symmetry of the solution, the remaining part of energy can be
  estimated as follows:
  \begin{equation}
    E[u(\cdot,t)]-E[u(\cdot,t);L]=\int_{L}^{\infty}(u_x^2(x,t)+2V(u(x,t)))dx.
  \end{equation}
  And by decrease of the solution for $x > 0$ we know that
  \begin{equation}
    \int_{L}^{\infty}u_x^2(x,t)dx\leq{u(L,t)\|u_x(x,t)\|_{L^{\infty}(\mathbb{R})}}.
  \end{equation}
  By standard parabolic regularity, for $t \geq 1$, the above
  expression converges to $0$ as $L\rightarrow\infty$. In addition, we
  have $E[v]-E[v;L]\rightarrow0$ as $L\rightarrow\infty$. So we only
  need to show that for any $\delta>0$ there exist a sufficiently
  large $L_{\delta}>0$, $T_{\delta}>0$ such that for any
  $t>T_{\delta}$,
  \begin{equation}
    \left| \int_{L_{\delta}}^{\infty}V(u(x,t))dx \right| <\delta.
  \end{equation}

  If $f'(0)<0$, then there exists $K>0$ such that $f(u)\leq-Ku$ for
  all $u\in[0,\theta_0/2]$. We can then finish the proof of the lemma
  by an $L^2$ decay estimate similar to the one in the proof of Lemma
  \ref{etozero}. Taking $L>0$ satisfying $v(L)<\theta_0/4$, there
  exists $T>0$ such that $u(x,t)<\theta_0/2$ for any $x\in(L,\infty)$
  and any $t>T$. Then for $t>T$ we have
  \begin{eqnarray}
    \frac{d}{dt}\int_{L}^{\infty}u^2dx
    &=&2\int_{L}^{\infty}u(x,t)(u_{xx}(x,t)+f(u))dx\nonumber\\
    &\leq&2u(L,t)|u_x(L,t)|-2K\int_{L}^{\infty}u^2(x,t)dx.
  \end{eqnarray}
  Since $\displaystyle \lim_{t\rightarrow\infty}u(L,t)|u_x(L,t)|=v(L)|v'(L)|$, from
  the above inequality and the relation $0\leq{V(u(x,t))}\leq{Cu^2}$
  on $u\in[0,\theta_0]$ for some $C>0$, we know that there exists
  $\hat{T}>T$ such that for any $t>\hat{T}$
\begin{equation}
  0\leq\int_{L}^{\infty}V(u(x,t))dx<\frac{2Cv(L)|v'(L)|}{K}.
\end{equation}
Since $v(L)v'(L)\rightarrow0$ as $L\rightarrow\infty$, we have the
desired conclusion.
\end{proof}

Now to the degenerate case.

\begin{lemma}
  \label{etoeo2}
  If $f'(0)=0$, then $\displaystyle
  \lim_{t\rightarrow\infty}u(x,t)=v(x)$ uniformly in $\mathbb{R}$
  implies $\displaystyle \lim_{t\rightarrow\infty}E[u(\cdot,t)]=E_0$,
  when (\ref{eq:fconc}) and (\ref{power}) hold.
\end{lemma}
\begin{proof}
  In the spirit of Lemma \ref{etoeo1}, we only need to show that
  \begin{equation}
    \limsup_{t \to \infty} \int^{\infty}_{L}V(u(x,t))dx\rightarrow0\;\;
    \text{as}\;L\rightarrow\infty.
  \end{equation}
  By (\ref{power}), for any sufficiently small $\delta>0$ we have
  \begin{equation}
    0\leq{V(u)}\leq{2 k u^{p+1} \over p+1}\qquad \forall u \in [0, \delta],
  \end{equation}
  Furthermore, by Proposition \ref{p:bump} we can fix $L \sim
  \delta^{-\frac{p - 1}{2}} \gg 1$ such that $v(L)=\delta/2$. Because
  $u(L,t)$ converges to $v(L)$ as $t\rightarrow\infty$, for
  sufficiently large $t$ we have $u(x, t) \leq \delta$ for all $x \geq
  L$ and
  \begin{equation}
    0 \leq \int^{\infty}_{L}V(u(x,t))dx
    \leq\frac{2 k}{p+1}\|u(\cdot,t)\|^{p+1}_{L^{p+1}(L,\infty)}.
  \end{equation}
  Then we only need to control $\|u(\cdot,t)\|_{L^{p+1}(L,\infty)}$ by
  $\delta$ for large enough $t$.

  We denote by $\bar{v}(x)$ a shift of the bump solution $v(x)$ from
  Proposition \ref{p:bump} which satisfies $0 < \bar v \leq \delta$
  for all $x>L$ and
  \begin{equation}\left\{\!\!\!
      \begin{array}{lll}
        0&=&\bar{v}''+f(\bar{v}),\;\;x>L,\\
        \bar v(L)&=&\delta, \\
        \bar v(\infty) & = & 0.
      \end{array}\right.
  \end{equation}
  Then we construct a supersolution $\bar{u}$, which solves the
  half-line problem:
  \begin{equation}\left\{\!\!\!
      \begin{array}{lll}
        \bar{u}_t&=&\bar{u}_{xx}+f(\bar{u}),\;\;x>L,\;t>T,\\
        \bar{u}(L,t)&=&\delta,\\
        \bar{u}(x,T)&=&\max\{u(x,T),\bar{v}(x)\}.
      \end{array}\right.
  \end{equation}
  Note that since $\hat u(x, t) \equiv \delta$ is a supersolution for
  $\bar u(x, t)$, we have $\bar u(x, t) \leq \delta$ for all $x \geq
  L$ and $t \geq T$.  And by comparison principle we have $u(x, t)
  \leq \bar u(x, t)$ for all $x \geq L$ and $t \geq T$.

  We now introduce
  \begin{equation}
    w(x,t) := \bar{u}(x,t)-\bar{v}(x)\geq0,\;\;x>L,\;t>T,
  \end{equation}
  which satisfies the linear equation:
  \begin{equation}
    w_t=w_{xx}+f'(\tilde w) w,\;\;x>L,\;t>T,
  \end{equation}
  for some $\bar{v}\leq\tilde{w}\leq\bar{u}$, with homogeneous
  Dirichlet boundary condition
  \begin{equation}
    w(L,t)=0,\;\;t>T.
  \end{equation}
  Since $0 \leq w(x,T) \leq u(x,T)$, we have $w(\cdot, T) \in L^2(\mathbb R)$ by
  Proposition \ref{p:exist}. Furthermore, in view of \eqref{eq:fconc}
  the solution $\bar w$ of the heat equation with the same initial and
  boundary conditions:
  \begin{align}
    \label{eq:wheat}
    \bar w_t = \bar w_{xx}, \; x > L, \; t > T, \qquad \bar w(L, t) =
    0, \; t > T, \qquad \bar w(x, T) = w(x, T), \; x > L,
  \end{align}
  is a supersolution for $w$. Then, by the estimate similar to the one
  in \eqref{eq:uto0} and comparison principle and we have for some $C
  > 0$:
  \begin{align}
    \label{eq:wto0}
    \| w(\cdot, t) \|_{L^{p+1}(\mathbb R)} \leq \| \bar w(\cdot, t)
    \|_{L^{p+1}(\mathbb R)} \leq C t^{-{p-1 \over 4 (p + 1)}} \|
    w(\cdot, T) \|_{L^2(\mathbb R)} \to 0 \; \text{as} \; t \to
    \infty.
  \end{align}

  Estimating $\| \bar u(\cdot, t) \|_{L^{p+1}(\mathbb R)}$ in terms of
  $\| w(\cdot, t) \|_{L^{p+1}(\mathbb R)}$, we obtain
  \begin{equation}
    \label{eq:ubarw}
    \|\bar{u}(\cdot,t)\|_{L^{p+1}(L,\infty)}
    \leq
    \|w(\cdot,t)\|_{L^{p+1}(L,\infty)}+\|\bar{v}\|_{L^{p+1}(L,\infty)}
    \qquad \forall t \geq T.
  \end{equation}
  On the other hand, it is clear that the estimates in Proposition
  \ref{p:bump} apply to $\bar v$ as well. Therefore
  \begin{equation}
    \label{eq:vLlarge}
    \|\bar{v}\|^{p+1}_{L^{p+1}(L,\infty)}  \leq C \delta^{{p+3 \over
        2}}.
  \end{equation}
  for some $C > 0$ and all $\delta > 0$ sufficiently small. Finally,
  combining \eqref{eq:wto0} and \eqref{eq:vLlarge} in
  \eqref{eq:ubarw}, by comparison principle we conclude that $\| u
  (\cdot,t)\|_{L^{p+1}(L,\infty)}$ can be made arbitrarily small for
  all $t \geq T$ by choosing a sufficiently small $\delta$ in the
  limit $t \to \infty$.
\end{proof}

Note that we have now proved our Theorem \ref{theorembistable}.

Let us finally consider the question of threshold phenomena. We use
similar notations as in \cite{DM2010}. Let $X :=
\{\phi(x):\phi(x)\;\text{satisfies (\ref{initial}) and (SD)}\}$. We
consider a one-parameter family of initial conditions
$\phi_{\lambda}$, $\lambda>0$, satisfying the following conditions:
\begin{enumerate}
\item[(P1)] For any $\lambda>0$, $\phi_{\lambda} \in X$, the map
  $\lambda\mapsto\phi_{\lambda}$ is continuous from $\mathbb{R_{+}}$
  to $L^2(\mathbb{R})$;
\item[(P2)] If $0<\lambda_1<\lambda_2$, then
  $\phi_{\lambda_1}\leq\phi_{\lambda_2}$ and
  $\phi_{\lambda_1}\not= \phi_{\lambda_2}$ in $L^2(\mathbb R)$.
\item[(P3)] $\displaystyle \lim_{\lambda\rightarrow0}\phi_{\lambda}(x)=0$
 in $L^2(\mathbb R)$.
\end{enumerate}
\noindent We denote by $u_\lambda(x, t)$ the solution of \eqref{main}
with the initial datum $\phi_\lambda$.

Here is our main result concerning threshold phenomena for bistable
nonlinearities.

\begin{thm}
  \label{sharpbistable}
  Under the same conditions as in Theorem \ref{thmbistable}, suppose
  that (P1) through (P3) hold. Then one of the following
  two conclusions is true:\\
  1. $\displaystyle \lim_{t\rightarrow\infty} u_\lambda(x, t) =0$ uniformly in
  $\mathbb{R}$ for every
  $\lambda>0$;\\
  2. There exists $\lambda^{\ast}>0$ such that
  \begin{equation*}
    \lim_{t\rightarrow\infty}u_{\lambda}(x,t)=\left\{\!\!\!
      \begin{array}{lll}
        0, & \text{uniformly in $\mathbb{R}$,} & \text{for
          $0\leq\lambda<\lambda^{\ast}$,}\\
        v(x), & \text{uniformly in $\mathbb{R}$,} & \text{for
          $\lambda=\lambda^{\ast}$,}\\
        1, & \text{locally uniformly in $\mathbb{R}$,} & \text{for
          $\lambda>\lambda^{\ast}$.}
      \end{array}
    \right.
  \end{equation*}
\end{thm}
\begin{proof}
  We define
$$\Sigma_0:= \{\lambda>0:\;u_{\lambda}(x,t)\rightarrow0
\;as\;t\rightarrow\infty\;\text{uniformly in}\;x\in\mathbb{R}\},$$
$$\Sigma_1 := \{\lambda>0:\;u_{\lambda}(x,t)\rightarrow1
\;as\;t\rightarrow\infty\;\text{locally uniformly
  in}\;x\in\mathbb{R}\}.$$ We know that $\lambda \in \Sigma_0$ if and
only if there exists $T \geq 0$ such that $u(0,T)<\theta_0$. Clearly
the set $\Sigma_0$ is open. Furthermore, by comparison principle, if
$\hat{\lambda}\in\Sigma_0$, then for any $\lambda<\hat{\lambda}$,
$\lambda\in\Sigma_0$. So $\Sigma_0$ is an open interval.

If $\Sigma_0\neq(0,\infty)$, then the set $\Sigma_1$ is an open
  interval (semi-infinite) as well. Indeed, by Corollary
  \ref{spreadcondition} for every $\lambda \in \Sigma_1$ there exists
  $T \geq 0$ such that $E[u_\lambda(\cdot, T)] < 0$. Then by
  continuity of the energy functional in $H^1(\mathbb R)$ and
  continuous dependence in $H^1(\mathbb R)$ of the solution at $t > 0$
  on the initial data in $L^2(\mathbb R)$ (see Proposition
  \ref{p:exist}), there exists $\delta > 0$ such that for all
  $|\lambda' - \lambda| < \delta$ we have $E[u_{\lambda'}(\cdot, T)] <
  0$. Hence $\lambda' \in \Sigma_1$ as well. And by comparison
principle, if $\tilde{\lambda}\in\Sigma_1$, then for any
$\lambda>\tilde{\lambda}$, $\lambda\in\Sigma_1$. Then we know that
$\mathbb{R_{+}}\setminus(\Sigma_0\cup\Sigma_1)$ is a closed set, and,
more precisely, a closed interval.

We will prove that if $\mathbb{R_{+}}\setminus(\Sigma_0\cup\Sigma_1)$
is not empty, then it contains only one point. Consider the
Schr\"odinger-type operator
\begin{equation}
  \mathfrak{L}=-\frac{d^2}{dx^2}+ V(x), \qquad V(x) :=-f'(v(x)),
\end{equation}
and the associated Rayleigh quotient (for technical background,
see, e.g., \cite[Chapter 11]{lieb-loss}):
\begin{equation}
  \mathfrak{R}(\phi) :=\frac{\int_\mathbb{R} \left( |\phi'|^2 + V(x)
        \phi^2 \right) dx}{\int_\mathbb{R}
      \phi^2 dx}.
\end{equation}
Since by Proposition \ref{p:bump} we have $v' \in H^1(\mathbb
  R)$, translational symmetry of the problem yields (weakly
  differentiate \eqref{stationary} and test with $v'$):
\begin{equation}
  \mathfrak{R}(v')=0.
\end{equation}
Furthermore, since the function $v'$ changes sign and 
$\lim_{|x| \to \infty} V(x) \geq 0$, $v'$ is not a
minimizer of $\mathfrak{R}$. Therefore, since $V - \displaystyle
\lim_{|x| \to \infty} V(x) \in L^1(\mathbb R)$ by Proposition
\ref{p:bump}, there exists a positive function $\phi_0 \in
H^1(\mathbb{R})$ that minimizes $\mathfrak{R}$, with $\min_{\phi \in
  H^1(\mathbb R)} \mathfrak{R}(\phi) =: \nu_0<0$.  Approximating
$\phi_0$ by a function with compact support and using it as a test
function, we can then see that $\min_{\phi \in H^1_0(-L, L)}
\mathfrak{R}(\phi) =: \nu^L_0<0$ as well for a sufficiently large
$L>0$, and in this case there exists a positive minimizer $\phi^L_0\in
H^1_0(-L,L) \cap C^2(-L,L) \cap C^1([-L,L])$ such that
\begin{equation}
  \mathfrak{L}(\phi^L_0)=\nu^L_0\phi^L_0.
\end{equation}

If $\Sigma_1$ is not empty and the threshold set
  $\mathbb{R_{+}}\setminus(\Sigma_0\cup\Sigma_1)$ does not contain
only one point, then there exist two distinct values
$0 < \lambda_1<\lambda_2$ in the threshold set. Since
$f(u) \in C^1([0, \infty))$, $f'(u)$ is uniformly
continuous on $[0, \max \{1, \|\phi\|_{L^{\infty}} \}]$.
Thus, there exists $\delta > 0$ such that
\begin{equation}
  |f'(u_1)-f'(u_2)|<\frac{|\nu^L_0|}{2},
\end{equation}
for any $u_1, u_2 \in [0, \max \{1, \|\phi\|_{L^{\infty}} \}]$
satisfying $|u_1-u_2| < \delta$. Since $\lambda_{1,2} \in
\mathbb{R_{+}}\setminus(\Sigma_0\cup\Sigma_1)$, we have
$\displaystyle \lim_{t\rightarrow\infty}u_{\lambda_{1,2}}(x,t)=v(x)$ uniformly in
$x\in{\mathbb{R}}$. Then, there exists $T$ sufficiently large, such
that $|u_{\lambda_{1,2}}(x,t) - v(x)| < \delta$ for any $t \geq {T}$
and all $x \in \mathbb{R}$. So we have
\begin{equation}
  \max_{x\in[-L,L]}|f'(v(x))-f'(\tilde u(x,t))|<\frac{|\nu^L_0|}{2},
\end{equation}
for every $u_{\lambda_1}(x,t) \leq \tilde{u}(x,t) \leq
u_{\lambda_2}(x,t)$ and all $t \geq T$. However, let
$w(x,t)=u_{\lambda_2}(x,t)-u_{\lambda_1}(x,t)$, then $w(x,t)$
satisfies the following equation,
\begin{equation}
  w_t=w_{xx}+f'(\tilde{u})w,\;\;x \in \mathbb{R},\;t>0,
\end{equation}
for some $u_{\lambda_1}(x,t) \leq \tilde{u}(x,t) \leq
u_{\lambda_2}(x,t)$. By the strong maximum principle $w(x,t)>0$ for
any $x \in \mathbb{R}$ and $t>0$. Hence there exists $\varepsilon>0$
such that $w(x,T)>\varepsilon\phi^L_0(x)$. Let
$\varepsilon\phi^L_0(x)=: \underline{w}(x,t)$. Then
\begin{eqnarray}
  \underline{w}_t-\underline{w}_{xx}-f'(\tilde{u})\underline{w}
  &=&-\underline{w}_{xx}-f'(v)\underline{w}
  +(f'(v)-f'(\tilde{u}))\underline{w}\nonumber\\
  &=&\nu^L_0\underline{w}+(f'(v)-f'(\tilde{u}))
  \underline{w}\nonumber\\
  &\leq&\frac{\nu^L_0}{2}\underline{w}\nonumber\\
  &\leq&0,
\end{eqnarray}
which implies that $\underline{w}(x,t)$ is a subsolution for
$t\geq{T}$. So by comparison principle
\begin{equation}
  u_{\lambda_2}-u_{\lambda_1}\geq\varepsilon\phi^L_0(x),\;\;{\forall}t\geq{T},
\end{equation}
i.e. there exists a barrier between $u_{\lambda_1}$ and
$u_{\lambda_2}$, which contradicts the assumption that both
$u_{\lambda_1}(x,t)$ and $u_{\lambda_2}(x,t)$ converge to $v(x)$
uniformly in $\mathbb{R}$, as $t\rightarrow\infty$. It means that if
$\mathbb{R_{+}}\setminus(\Sigma_0\cup\Sigma_1)$ is not empty, then it
only contains one point.
\end{proof}

\begin{rmk}
  By Corollary \ref{spreadcondition} and comparison principle, to
  ensure that $\lambda^* < \infty$ in Theorem \ref{sharpbistable} it
  is enough if there exists $\lambda > 0$ and $\tilde{\phi}_\lambda
  \in L^2(\mathbb R)$ such that $0 \leq \tilde{\phi}_\lambda \leq
  \phi_\lambda$ and $E[\tilde{\phi}_\lambda] < 0$. This condition is
  easily seen to be verified for the family of characteristic
  functions of growing symmetric intervals studied by Kanel' \cite{K1964}. Also, by
  Theorem \ref{theorembistable} and the monotone decrease of the
  energy evaluated on solutions the condition $E[\phi_\lambda] < E_0$
  for some $\lambda > 0$ implies that $u_\lambda(x, t) \not\to
  v(x)$. In particular, if $\sup_{0 < \lambda < \bar\lambda}
  E[\phi_\lambda] < E_0$, then $\lambda^* > \bar\lambda$.
\end{rmk}

\section{Monostable Nonlinearity}
\label{s:mono}

In this section, we study the monostable nonlinearity,
i.e. $f(u)\in{C^1([0,\infty),\mathbb{R})}$,
\begin{equation}
  \label{monostable}
  f(0)=f(1)=0, \quad f(u)\left\{\!\!\!
    \begin{array}{ll}
      >0, & in \; (0,1),\\
      <0, & in \; (1,\infty).
    \end{array}\right.
\end{equation}
Moreover, we assume that the monostable nonlinearity $f(u)$ also
satisfies
\begin{equation}
  \label{monostablezero}
  f'(0)=0.
\end{equation}
Typical examples are the Arrhenius combustion nonlinearity
\begin{equation}
  f(u)=(1-u)e^{-\frac{a}{u}},\;\; a
  >0,
\end{equation}
and the generalized Fisher nonlinearity, i.e. the nonlinearity
\begin{equation}
  \label{fisher}f(u)=u^p (1-u),
\end{equation}
with exponent $p>1$.

Under conditions (\ref{monostable}), there exists one root of $V(u)$:
$u=0$, and possibly a second root $u=\theta^{\diamond}>1$. However,
since
$\displaystyle \lim_{t\rightarrow\infty}\|u(x,t)\|_{L^{\infty}(\mathbb{R})}\leq1$,
without loss of generality, we suppose that
$\|\phi\|_{L^{\infty}(\mathbb{R})}<\theta^{\diamond}$. So that we
always suppose that $V(u)\leq0$.

We have the following theorems about convergence and one-to-one
relations between the limit value of the energy and the long time
behavior of solutions, similar to the bistable case.

\begin{thm}
  \label{thmmonostable}
  Let $f$ satisfy conditions (\ref{monostable}) and
  (\ref{monostablezero}), and let $\phi(x)$ satisfy condition
  (\ref{condition}) and hypothesis (SD). Then one of the following
  holds.
  \begin{enumerate}
  \item $\displaystyle \lim_{t\rightarrow\infty}u(x,t)=1$ locally uniformly in
    $\mathbb{R}$,
  \item $\displaystyle \lim_{t\rightarrow\infty}u(x,t)=0$ uniformly in
    $\mathbb{R}$.
  \end{enumerate}
\end{thm}

\begin{thm}
  \label{theoremmonostable}
  Under the same conditions as in Theorem \ref{thmmonostable}, we have
  the following one-to-one relation.
  \begin{enumerate}
  \item $\displaystyle \lim_{t\rightarrow\infty}u(x,t)=1$ locally uniformly in
    $\mathbb{R}$ ${\Leftrightarrow}$
    $\displaystyle \lim_{t\rightarrow\infty}E[u(\cdot,t)]=-\infty$.
  \item $\displaystyle \lim_{t\rightarrow\infty}u(x,t)=0$ uniformly in
    $\mathbb{R}$ ${\Leftrightarrow}$ $\displaystyle
    \lim_{t\rightarrow\infty}E[u(\cdot,t)]=0$.
  \end{enumerate}
\end{thm}

Throughout the rest of this section, the hypotheses of Theorem
\ref{thmmonostable} are assumed to be satisfied. We start by
establishing the following conclusion.

\begin{lemma}
  \label{l:E0l0}
  If $\displaystyle \lim_{t\rightarrow\infty}u(x,t)=1$ locally
  uniformly in $\mathbb{R}$, then $\displaystyle
  \lim_{t\rightarrow\infty}E[u(\cdot,t)]=-\infty$. And if
  $\displaystyle \lim_{t\rightarrow\infty}u(x,t)=0$ uniformly in
  $\mathbb{R}$, then $\displaystyle
  \lim_{t\rightarrow\infty}E[u(\cdot,t)]\leq0$.
\end{lemma}
\begin{proof}
  Under condition (\ref{monostable}), we know that
  $\int_{\mathbb{R}}V(u)dx \leq 0$. And if $u\rightarrow1$ locally
  uniformly in $\mathbb{R}$, then
\begin{equation}
  \lim_{t\rightarrow\infty}\int_{\mathbb{R}}V(u(x,t))dx=-\infty.
\end{equation}
By hypothesis (SD), we have
\begin{eqnarray}
  \int_{\mathbb{R}}\frac{1}{2}u^2_x(x,t)dx&=&\int_0^{\infty}u^2_x(x,t)dx
  \nonumber\\
  &\leq&\|u_x(x,t)\|_{L^{\infty}(\mathbb{R})} u(0,t). \label{eq:ux2}
\end{eqnarray}
Then by standard parabolic regularity the left-hand side of
\eqref{eq:ux2} is bounded uniformly in time. So we proved the first
conclusion. On the other hand, if $u\rightarrow0$ uniformly in
$\mathbb{R}$, then the left-hand side of \eqref{eq:ux2} converges to
$0$. In view of $V(u(x, t)) \leq 0$, we proved the second conclusion.
\end{proof}

Similarly to the Lemma \ref{wavelike} for the bistable case, we have
the following lemma for the monostable case.

\begin{lemma}
  \label{wavelikemonostable}
  Assume that there exists $c_0>0$ such that $\phi(x) \in
  H^1_{c_0}(\mathbb{R})$. If there exists $T \geq 0$ such that
  $E[u(\cdot,T)]<0$, then $u(x,t)$ is wave-like.
\end{lemma}
\begin{proof}
  Since $\phi(x) \in H_{c_0}^1(\mathbb{R})$, we have $u(x,T) \in
  H^1(\mathbb{R})\cap{H_{c_0}^1(\mathbb{R})}$. For any small
  $\varepsilon>0$, when $E[u(x,T)]=-\varepsilon<0$, there exists $L>0$
  such that
  \begin{equation}
    0\leq \frac{1}{2} \int_{L}^{\infty}e^{c_{0}x}u_x^2(x,T) dx
    <\frac{\varepsilon}{8},
  \end{equation}
  \begin{equation}
    -\frac{\varepsilon}{8}<\int_L^\infty e^{c_0x}V(u(x,T))dx\leq0.
  \end{equation}
  Note that if we use smaller $c \geq 0$ instead of $c_0$ in the above
  inequalities, they still hold. And by the definition of $L$ we know
  that
\begin{equation}
  \int_{-L}^L
  \left( \frac{1}{2}u_x^2(x,T)+V(u(x,T)) \right) dx<-\frac{3\varepsilon}{4}.
\end{equation}
So we can find a sufficiently small $c > 0$ such that $c < c_0$ and
\begin{equation}
  \int_{-L}^L e^{c x} \left( \frac{1}{2}u_x^2(x,T)
    +V(u(x,T)) \right) dx<-\frac{\varepsilon}{2}.
\end{equation}
Then we have
\begin{equation}
  \Phi_c[u(\cdot,T)]=\int_{\mathbb{R}}e^{c x}
  \left( \frac{1}{2}u_x^2(x,T)+V(u(x,T)) \right) dx<0.
\end{equation}
So $u$ is wave-like.
\end{proof}

In contrast to the bistable case, for monostable case boundedness of
energy always implies extinction.

\begin{lemma}
  \label{tozeromonostable}
  Suppose that $E[u(\cdot,t)]$ is bounded from below for all $t\geq1$,
  then $\displaystyle \lim_{t\rightarrow\infty}u(x,t)=0$ uniformly in
  $\mathbb{R}$.
\end{lemma}
\begin{proof}
  Since the unique root of $V(u)$ is $0$, arguing as in Lemma
  \ref{origin} we know that $u(0,t)\rightarrow0$ as
  $t\rightarrow\infty$. Then we prove this lemma by using Proposition
  \ref{holder}.
\end{proof}

\begin{lemma}
  \label{l:El0}
  Suppose that there exists $T\geq0$ such that $E[u(\cdot,T)]<0$. Then
  $\displaystyle \lim_{t\rightarrow\infty}u(x,t)=1$ locally uniformly in
  $\mathbb{R}$.
\end{lemma}

\begin{proof}
  The proof is similar to the proof of Lemma \ref{test}. If
  $E[u(\cdot,T)]<0$ for some $T\geq0$, then there exists a
  sufficiently small $c > 0$ such that $\Phi_c[\varphi_L
  u(\cdot,T)]<0$ for large enough $L > 0$, where the cutoff function
  $\varphi_L$ is as in Lemma \ref{test}. And by the conditions
  (\ref{monostable}) and (\ref{monostablezero}), there exists 
  $\delta_0>0$, such that condition (\ref{largepotential}) holds. 
  Then from Proposition \ref{p:muratov}, we know that $R_{\delta_0} (t) >ct+R_0$
  for some $R_0 \in \mathbb R$. Similarly to Lemma \ref{spread}, since the unique
  solution of equation (\ref{stationary}) larger than $\delta_0$ is
  $1$ in the whole of $\mathbb R$, we conclude that
  $\displaystyle \lim_{t\rightarrow\infty}u(x,t)=1$ locally uniformly in
  $\mathbb{R}$.
\end{proof}

\noindent An immediate consequence of Lemma \ref{l:E0l0} and Lemma
\ref{l:El0} is the following.

\begin{cor}
  Suppose that $\displaystyle \lim_{t\rightarrow\infty}u(x,t)=0$ uniformly in
  $\mathbb{R}$, then $\displaystyle \lim_{t\rightarrow\infty}E[u(\cdot,t)]=0$.
\end{cor}

\noindent We have thus established Theorems \ref{thmmonostable} and
\ref{theoremmonostable}.

Our last theorem in this section concerns with the threshold phenomena
for monostable nonlinearities.

\begin{thm}
  \label{sharpmonostable}
  Under the same conditions as in Theorem \ref{thmmonostable}, suppose
  that (P1) through (P3) hold. Then one of the following holds:
  \begin{enumerate}
  \item $\displaystyle \lim_{t\rightarrow\infty} u_\lambda(x, t) =0$
    uniformly in $x \in \mathbb{R}$ for every $\lambda>0$;
  \item $\displaystyle \lim_{t\rightarrow\infty} u_\lambda(x, t) = 1$ locally
    uniformly in $x \in \mathbb{R}$ for every $\lambda>0$;
  \item There exists $\lambda^{\ast} > 0$ such that
    \begin{equation*}
      \lim_{t\rightarrow\infty}u_{\lambda}(x,t)=\left\{\!\!\!
        \begin{array}{lll}
          0, & \text{uniformly in $x \in \mathbb{R}$,} & \text{for
            $0 < \lambda\leq\lambda^{\ast}$,} \\
          1, & \text{locally uniformly in $x \in \mathbb{R}$,} & \text{for
            $\lambda>\lambda^{\ast}$.}
        \end{array}\right.
    \end{equation*}
  \end{enumerate}
\end{thm}
\begin{proof}
  Similarly to the proof of Theorem \ref{sharpbistable}, if neither
  $\Sigma_0 = \varnothing$ nor $\Sigma_1 = \varnothing$, then
  $\Sigma_1$ is an open interval. The conclusion then follows.
\end{proof}

Note that our sharp transition result above is nontrivial, e.g., for
the generalized Fisher nonlinearity in \eqref{fisher} with $p>p_c$,
where $p_c = 3$ is the Fujita exponent (see, e.g., \cite[Theorem
3.2]{AW1978}).

\section{Ignition Nonlinearity}
\label{s:ign}

The ignition nonlinearity $f(u)\in{C^1([0,\infty),\mathbb{R})}$
satisfies
\begin{equation}\label{ignition}
f(u)\left\{\!\!\!
\begin{array}{lll}
=0, & in \; [0,\theta_0]\cup\{1\},\\
>0, & in \; (\theta_0,1),\\
<0, & in \; (1,\infty),
\end{array}\right.
\end{equation}
for some $\theta_0\in(0,1)$. We also suppose that there exists
$\delta>0$ such that
\begin{equation}
  \label{convex}
  f(u)~ \text{is convex on} ~[\theta_0,\theta_0+\delta].
\end{equation}
Under \eqref{condition} and (\ref{ignition}), except on the interval
$[0,\theta_0]$, there exists at most one root $u=\theta^{\diamond}>1$
of $V(u)$. However, since
$\displaystyle \limsup_{t\rightarrow\infty}\|u(x,t)\|_{L^{\infty}(\mathbb{R})}\leq1$,
without loss of generality, we suppose that
$\|\phi\|_{L^{\infty}(\mathbb{R})}<\theta^{\diamond}$. So that we
always have $V(u)\leq0$.

Here are our main results concerning the long time behavior of
solutions and their energy.

\begin{thm}
  \label{thmignition}
  Let $f$ satisfy conditions (\ref{ignition}) and (\ref{convex}). Let
  $\phi(x)$ satisfy condition (\ref{condition}) and hypothesis
  (SD). Then one of the following holds.
  \begin{enumerate}
  \item $\displaystyle \lim_{t\rightarrow\infty}u(x,t)=1$ locally
    uniformly in $x \in \mathbb{R}$,
  \item $\displaystyle \lim_{t\rightarrow\infty}u(x,t)=\theta_0$
    locally uniformly in $x \in \mathbb{R}$,
  \item $\displaystyle \lim_{t\rightarrow\infty}u(x,t)=0$ uniformly in
    $x \in \mathbb{R}$.
  \end{enumerate}
\end{thm}

\begin{thm}
  \label{theoremignition}
  Under the same assumptions as in Theorem \ref{thmignition}, we have
  the following one-to-one relation.
  \begin{enumerate}
  \item $\displaystyle \lim_{t\rightarrow\infty}u(x,t)=1$ locally uniformly in
    $x \in \mathbb{R}$ ${\Leftrightarrow}$
    $\displaystyle \lim_{t\rightarrow\infty}E[u(\cdot,t)]=-\infty$.
  \item $\displaystyle \lim_{t\rightarrow\infty}u(x,t)=\theta_0$
    locally uniformly in $x \in \mathbb{R}$ or $\displaystyle
    \lim_{t\rightarrow\infty}u(x,t)=0$ uniformly in $x \in \mathbb{R}$
    ${\Leftrightarrow}$ $\displaystyle
    \lim_{t\rightarrow\infty}E[u(\cdot,t)]=0$.
  \end{enumerate}
\end{thm}

We prove the above theorems via a sequence of lemmas.

\begin{lemma}
  \label{l:Enegign}
  Suppose that there exists $T\geq0$ such that $E[u(\cdot,T)]<0$, then
  $\displaystyle \lim_{t\rightarrow\infty}u(x,t)=1$ locally uniformly in
  $\mathbb{R}$.
\end{lemma}
\begin{proof}
  The arguments follow those in the proof of Lemma \ref{test}. If
  $E[u(\cdot,T)]<0$ for some $T\geq0$, then there exists a
  sufficiently small $c > 0$ such that $\Phi_c[\varphi_L
  u(\cdot,T)]<0$ for large enough $L > 0$, where the cutoff function
  $\varphi_L$ is as in Lemma \ref{test}. And by the condition
  (\ref{ignition}), there exists $\delta_0>\theta_0$, such that
  condition (\ref{largepotential}) holds. Then from Proposition
  \ref{p:muratov}, we know that $R_{\delta_0} (t) >ct+R_0$ for some $R_0
  \in \mathbb R$. Similarly to Lemma \ref{spread}, since the unique
  solution of equation (\ref{stationary}) larger than $\delta_0$
  is $1$ in the whole of $\mathbb R$, we conclude that
  $\displaystyle \lim_{t\rightarrow\infty}u(x,t)=1$ locally uniformly in
  $\mathbb{R}$.
\end{proof}

\begin{lemma}
  Suppose that $E[u(\cdot,t)]$ is bounded from below in $t$. Then
  either $\displaystyle \lim_{t\rightarrow\infty}u(x,t)=0$ uniformly
  in $\mathbb{R}$, or $\displaystyle
  \lim_{t\rightarrow\infty}u(x,t)=\theta_0$ locally uniformly in
  $\mathbb{R}$.
\end{lemma}

\begin{proof}
  Same as in Lemma \ref{origin}, there exists an unbounded
    increasing sequence $\{t_n \}$ such that
\begin{equation}
  \lim_{n\rightarrow\infty}u(0,t_n) = \alpha,
\end{equation}
for some $\alpha\in[0,\theta_0]$. And in the spirit of Lemma
\ref{tobump} we have
\begin{equation}
  \lim_{n\rightarrow\infty}u(x,t_n)\equiv\alpha,
\end{equation}
locally uniformly in $x \in \mathbb R$. We need to prove that $\alpha$
is either $0$ or $\theta_0$. We argue by contradiction. Assume that
$0<\alpha<\theta_0$, then there exists $T \geq 0$ sufficiently large
such that $u(0,T)<\theta_0$. And for any $t>T$, $u(x,t)\equiv\theta_0$
is a supersolution of (\ref{main}), so that
$0\leq{u(x,t)}\leq\theta_0$ uniformly in $\mathbb{R}$. From the
definition of $f(u)$, it then implies that equation (\ref{main})
becomes
\begin{equation}
  u_t(x,t)=u_{xx}(x,t),
\end{equation}
for any $t>T$. However, the $L^2$ norm of the solution of the heat
equation is non-increasing, contradicting the assumption that $u(x,t)$
converges to $\alpha$ locally uniformly. So either $\alpha = 0$ or
  $\alpha = \theta_0$, which proves the lemma.
\end{proof}

\begin{cor}
  Suppose that $\displaystyle \lim_{t\rightarrow\infty}u(x,t)=1$
  locally uniformly in $\mathbb{R}$, then $\displaystyle
  \lim_{t\rightarrow\infty}E[u(\cdot,t)]=-\infty$.
\end{cor}

\begin{lemma}
  Both $\displaystyle \lim_{t\rightarrow\infty}u(x,t)=0$ uniformly in
  $\mathbb{R}$ and $\displaystyle
  \lim_{t\rightarrow\infty}u(x,t)=\theta_0$ locally uniformly in
  $\mathbb{R}$ imply $\displaystyle
  \lim_{t\rightarrow\infty}E[u(\cdot,t)]=0$.
\end{lemma}
\begin{proof}
  By Lemma \ref{l:Enegign}, $E[u(\cdot,t)] \geq 0$ for these
  behaviors. And since $V(u) \leq 0$ for any $u$, we have
\begin{equation}
  \label{gradient}
  E[u(\cdot,t)]\leq\int_{\mathbb{R}}\frac{1}{2}u^2_x(x,t)dx.
\end{equation}
So we only need to prove that the right-hand side of (\ref{gradient})
converges to $0$ as $t\rightarrow\infty$. From (SD), we have
\begin{equation}
  \int_{\mathbb{R}}\frac{1}{2}u^2_x(x,t)dx=\int_{0}^{\infty}u^2_x(x,t)dx
  \leq  \| u_x(\cdot,t) \|_{L^\infty(\mathbb R)} u(0,t).
\end{equation}
We are done if $\displaystyle \lim_{t\rightarrow\infty}u(x,t)=0$
uniformly in $\mathbb{R}$, because $\| u_x(\cdot,t)
\|_{L^\infty(\mathbb R)} $ is bounded by standard parabolic
regularity. So we only need to prove that $\| u_x(\cdot,t)
\|_{L^\infty(\mathbb R)} \to 0$ as $t \to \infty$ for the case
$\displaystyle \lim_{t\rightarrow\infty}u(x,t)=\theta_0$ locally
uniformly in $\mathbb{R}$. We know that $|u_{xx}(x,t)|$ is uniformly
bounded for all $x \in \mathbb R$ and all $t \geq 1$. Using the
convergence result $\displaystyle
\lim_{t\rightarrow\infty}u(0,t)=\theta_0$, by standard parabolic
regularity we also know that
\begin{equation}
  \lim_{t\rightarrow\infty}\sup_{|x|\leq{R_{\theta_0}(t)}}|u_x(x,t)|=0.
\end{equation}
Multiplying (\ref{main}) by $u_x$ and integrating from the leading
edge $R_{\delta}(t)$ to $\infty$, which is justified by
  Proposition \ref{p:exist}, for any $\theta \in (0, \theta_0]$, we
have
\begin{equation}
  \int^{\infty}_{R_{\theta}(t)}u_x(x,t)u_t(x,t)dx=
  \int^{\infty}_{R_{\theta}(t)}u_x(x,t)u_{xx}(x,t)dx,
\end{equation}
since $f(u)=0$ for any $u\in[0,\theta_0]$. Integrating by part and
applying Cauchy-Schwarz inequality, we obtain
\begin{eqnarray}
  \frac{1}{2}u^2_x(R_{\theta}(t),t)& \leq &
  \left( \int^{\infty}_{R_{\theta}(t)}u^2_x(x,t)dx\int^{\infty}_{R_{\theta}(t)
    }u^2_t(x,t)dx \right)^{\frac{1}{2}}\nonumber\\
  & \leq & \left(
    {\theta}\max_{x{\in}\mathbb{R}}|u_x(x,t)|\int^{\infty}_{R_{\theta}(t)
    }u^2_t(x,t)dx \right)^{\frac{1}{2}}\nonumber\\
  & \leq &\left(
    {\theta_0}\max_{x{\in}\mathbb{R}}|u_x(x,t)|\int_{\mathbb{R}}
    u^2_t(x,t)dx \right)^{\frac{1}{2}}.
\end{eqnarray}
Since $E[u(\cdot,t)]$ is bounded from below in $t$, there exists an
increasing sequence $\{t_n\}$ such that
$\displaystyle \lim_{n\rightarrow\infty}t_n=\infty$
and
\begin{equation}
  \lim_{n\rightarrow\infty}\int_{\mathbb{R}}u^2_t(x,t_n)dx=0.
\end{equation}
In turn, since $\theta$ is arbitrary in $(0,\theta_0]$, we have
\begin{equation}
  \lim_{n\rightarrow\infty}\sup_{x>R_{\theta_0}(t_n)}|u_x(x,t_n)|=0.
\end{equation}
This means that the right-hand side of (\ref{gradient}) converges to
$0$ on sequence $\{t_n\}$. The statement of the lemma then follows,
since $E[u(\cdot,t)]$ is non-increasing.
\end{proof}

We have now proved our convergence and equivalence theorems for the
ignition nonlinearity.  Studying the threshold phenomena for ignition
nonlinearity is a little different from the situation with bistable
nonlinearity. The main difficulty is to show that the threshold set
contains only a single point, since we cannot construct the type of
barrier used in the proof of Theorem \ref{sharpbistable}. Instead we
modify the proof by Zlato\v{s} in \cite{Z2006}, which uses a rescaling
technique for dealing only with the initial condition in the form of a
characteristic function.

\begin{lemma}
  \label{modify}
  Let $f:[0,\infty)\rightarrow\mathbb{R}$ be a Lipschitz function with
  $f(0)=0$. Let
  $U(x,t):\mathbb{R}\times[0,\infty)\rightarrow[0,\infty)$ be a
  classical solution of
\begin{equation}
  \label{maineq}
  U_t=U_{xx}+f(U),
\end{equation}
which is uniformly continuous up to $t = 0$. Denote by $U_1(x,t)$ and
$U_2(x,t)$ the solutions of equation (\ref{maineq}) with initial
conditions $U_1(x,0)$ and $U_2(x,0)$, respectively, and assume
$0\leq{U_1(x,0)}\leq{U_2(x,0)}$ for any $x\in\mathbb{R}$, and
$U_1(x_0,0)<U_2(x_0,0)$ for some $x_0\in\mathbb{R}$. Assume also that
for any $\rho>0$ the set
$\Omega_{0,\rho}=\{x\in\mathbb{R}:U_2(x,0)\geq\rho\}$ is
compact. Finally, assume that there are $0<\theta_1<\theta_2$ and
$\varepsilon_1>0$ such that for any $\theta\in[\theta_1,\theta_2]$ and
$\varepsilon\in[0,\varepsilon_1]$, we have
\begin{equation}
  \label{behavior}f(\theta+\varepsilon(\theta-\theta_1))
  \geq(1+\varepsilon)f(\theta),
\end{equation}
and assume $\|U_1 \|_{L^\infty(\mathbb R \times (0,
  \infty))}<\theta_2$ for any $t\in[0,\infty)$. Then
\begin{equation}
  \label{contradiction}
  \liminf_{t\rightarrow\infty}
  \inf_{U_1(x,t)>\theta_1}\frac{U_2(x,t)-\theta_1}{U_1(x,t)-\theta_1}>1,
\end{equation}
with the convention that the infimum over an empty set is $\infty$.
\end{lemma}

\begin{proof}
  It is essentially Lemma 4 of \cite{Z2006}.
\end{proof}

\begin{thm}
  \label{sharpignition}
  Under the same conditions as in Theorem \ref{thmignition}, suppose
  that (P1) through (P3) hold. Then one of the following holds:
  \begin{enumerate}
  \item $\displaystyle \lim_{t\rightarrow\infty} u_\lambda(x, t) =0$
    uniformly in $x \in \mathbb{R}$ for every $\lambda>0$;
  \item There exists $\lambda^{\ast}>0$ such that
    \begin{equation*}
      \lim_{t\rightarrow\infty}u_{\lambda}(x,t)=\left\{\!\!\!
        \begin{array}{lll}
          0, & \text{uniformly in $x \in \mathbb{R}$,} & \text{for
            $0 < \lambda<\lambda^{\ast}$,}\\
          \theta_0, & \text{locally uniformly in $x \in \mathbb{R}$,} &
          \text{for $\lambda=\lambda^{\ast}$,}\\
          1, & \text{locally uniformly in $x \in \mathbb{R}$,} & \text{for
            $\lambda>\lambda^{\ast}$.}
        \end{array}\right.
    \end{equation*}
  \end{enumerate}
\end{thm}

\begin{proof}
  Similarly to the proof of Theorem \ref{sharpbistable}, we can show
  that if $\Sigma_0\neq(0,\infty)$, then both $\Sigma_0$ and
  $\Sigma_1$ are open intervals, and hence
  $\mathbb{R}_{+}\setminus(\Sigma_0\cup\Sigma_1)$ is a closed
  interval. Then we only need to prove that
  $\mathbb{R}_{+}\setminus(\Sigma_0\cup\Sigma_1)$ contains only a
  single point.  We need to verify that if $f(u)$ satisfies
  (\ref{ignition}) and (\ref{convex}), then there exists
  $\varepsilon_1>0$, and $0<\theta_1<\theta_0<\theta_2<1$ such that
  condition (\ref{behavior}) holds. Note that convexity of $f(u)$ on
  $[\theta_0,\theta_0+\delta]$ implies that $f(u)$ is nondecreasing on
  $[\theta_0,\theta_0+\delta]$, and $\theta_0+\delta<1$. Taking
  $\varepsilon_1=\delta/2$, $\theta_1=\theta_0/2$,
  $\theta_2=(3\theta_0+\delta)/3$, we only need to prove that
  (\ref{behavior}) holds for any $\varepsilon\in[0,\varepsilon_1]$ and
  $\theta\in[\theta_0,\theta_2]$. Let $\alpha :=
  \theta-\theta_0\in[0,\delta/3]$. We have the following estimate of
  the left-hand side of (\ref{behavior}),
\begin{equation}
  f(\theta+\varepsilon(\theta-\theta_1))=
  f\left( \theta_0+(1+\varepsilon)\alpha+\frac{\varepsilon\theta_0}{2}
  \right)
  \geq{f(\theta_0+(1+\varepsilon)\alpha)},
\end{equation}
since $\theta+\varepsilon(\theta-\theta_1)<\theta_0+\delta$. By
convexity we also have
\begin{equation}
  f(\theta_0+\alpha)\leq
  \frac{f(\theta_0+(1+\varepsilon)\alpha)}{1+\varepsilon}
  +\frac{{\varepsilon}f(\theta_0)}{1+\varepsilon},
\end{equation}
which proves
\begin{equation}
  f(\theta_0+(1+\varepsilon)\alpha)\geq(1+\varepsilon)f(\theta_0+\alpha),
\end{equation}
for any $\varepsilon\in[0,\varepsilon_1]$ and
$\alpha\in[0,\theta_2-\theta_0]$. Hence (\ref{behavior}) holds for any
$\theta\in[\theta_1,\theta_2]$ and
$\varepsilon\in[0,\varepsilon_1]$.

Then we suppose that there exist two distinct values $0 < \lambda_1 <
\lambda_2$ in the threshold set
$\mathbb{R}_{+}\setminus(\Sigma_0\cup\Sigma_1)$. Denote by
$u_{\lambda_1}(x,t)$ and $u_{\lambda_2}(x,t)$ these solutions with
initial conditions $\phi_{\lambda_1}$ and $\phi_{\lambda_2}$,
respectively. Taking $\theta_1$, $\theta_2$ as above, there exists
$T>0$ such that $\|u_{\lambda_1}\|_{L^\infty(\mathbb R \times (0,
  \infty))} <\theta_2$, for any $t\geq{T}$. And for any $t\geq{T}$,
let $U_1(x,t) := {u_{\lambda_1}(x,t-T)}$ and $U_2(x,t) :=
{u_{\lambda_2}(x,t-T)}$. Obviously all the assumptions of Lemma
\ref{modify} hold. So there exists $r>1$ such that
\begin{equation}
  \label{limit}
  \liminf_{t\rightarrow\infty}
  \frac{U_1(0,t)-\theta_1}{U_2(0,t)-\theta_1}\geq{r}.
\end{equation}
But both $U_1(0,t)$ and $U_2(0,t)$ converge to $\theta_0$ as
$t\rightarrow\infty$. So that the left-hand side of (\ref{limit}) must
be $1$, which is a contradiction.
\end{proof}

\begin{rmk}
  The $C^1$ property of $f(u)$ and condition (\ref{ignition})
  imply that $f'(\theta_0)=0$. If we suppose that $f(u)\in C[0,\infty)
  \cap C^1(\theta_0,\infty)$, together with (\ref{ignition}), and
  $\displaystyle \lim_{u\rightarrow\theta_0^{+}}f'(u)>0$, then without local
  convexity condition (\ref{convex}) all the conclusions about
  convergence, equivalence, and sharp transition in this section still
  hold.
\end{rmk}

\section*{Acknowledgements}

  This work was supported by NSF via grants DMS-0718027, DMS-0908279
  and DMS-1119724. The authors would are grateful to P. Gordon,
  H. Matano, V. Moroz and M. Novaga for valuable discussions.

\bibliography{references}
\bibliographystyle{plain}
\end{document}